%% file: main.tex
\documentclass{article}

\usepackage{microtype}
\usepackage[nottoc]{tocbibind}
\usepackage[utf8]{inputenc}
\usepackage{graphicx}
\usepackage{float}
\usepackage{bm}
\usepackage{amsthm}
\usepackage{amsmath}
\usepackage{amssymb}
\usepackage{mathrsfs}
\usepackage{dsfont}
\usepackage[all]{xy}
\usepackage{tikz-cd}
\usepackage{stmaryrd}
\usepackage{enumitem}
\usepackage{tabularx}
\usepackage{placeins}
\usepackage{hyperref}

\theoremstyle{definition}
\newtheorem{Definition}{Definition}[subsection]

\theoremstyle{plain}
\newtheorem{Theorem}[Definition]{Theorem}

\theoremstyle{plain}

\theoremstyle{plain}
\newtheorem{Proposition}[Definition]{Proposition}

\theoremstyle{plain}
\newtheorem{Lemma}[Definition]{Lemma}

\theoremstyle{plain}
\newtheorem{Corollary}[Definition]{Corollary}

\theoremstyle{plain}

\theoremstyle{plain}

\theoremstyle{plain}

\theoremstyle{definition}
\newtheorem{Construction}[Definition]{Construction}

\theoremstyle{definition}
\newtheorem{Example}[Definition]{Example}

\theoremstyle{definition}

\theoremstyle{remark}
\newtheorem{Remark}[Definition]{Remark}

\theoremstyle{plain}
\newcommand{\thistheoremname}{}
\newtheorem*{genericthm*}{\thistheoremname}
\newenvironment{namedthm*}[1]
  {\renewcommand{\thistheoremname}{#1}%
   \begin{genericthm*}}
  {\end{genericthm*}}

\title{Drinfeld Centers and Morita Equivalence Classes of Fusion 2-Categories}
\author{Thibault D. Décoppet}
\date{October 2023}

\begin{document}

\bibliographystyle{alpha}

    \maketitle
    \hspace{1cm}
    \begin{abstract}
        We prove that the Drinfeld center of a fusion 2-category is invariant under Morita equivalence. We go on to show that the concept of Morita equivalence between connected fusion 2-categories corresponds to a notion of Witt equivalence between braided fusion 1-categories. A strongly fusion 2-category is a fusion 2-category whose braided fusion 1-category of endomorphisms of the monoidal unit is $\mathbf{Vect}$ or $\mathbf{SVect}$. We prove that every fusion 2-category is Morita equivalent to the 2-Deligne tensor product of a strongly fusion 2-category and an invertible fusion 2-category. We proceed to show that every fusion 2-category is Morita equivalent to a connected fusion 2-category. As a consequence, we find that every rigid algebra in a fusion 2-category is separable. This implies in particular that every fusion 2-category is separable. Conjecturally, separability ensures that a fusion 2-category is 4-dualizable. We define the dimension of a fusion 2-category, and prove that it is always non-zero. Finally, we show that the Drinfeld center of any fusion 2-category is a finite semisimple 2-category.
    \end{abstract}
    
\tableofcontents

\section*{Introduction}
\addcontentsline{toc}{section}{Introduction}

Given a monoidal 1-category $\mathcal{C}$, one can form its Drinfeld center $\mathcal{Z}(\mathcal{C})$, that is the braided monoidal 1-category whose objects are pairs consisting of an object of $\mathcal{C}$ together with a coherent half-braiding. This construction categorifies the notion of the center of an algebra, and was first considered by Drinfeld in unpublished notes as an abstraction of the double construction for Hopf algebras \cite{Dri}. The construction of the Drinfeld center as a braided monoidal 1-category subsequently appeared in the category theory literature in \cite{JS}, and its relation with the double of a Hopf algebra was made precise in \cite{Ma}. The importance of the Drinfeld center in quantum algebra was then reinforced in \cite{Mu2}. Namely, it was shown by M\"uger that the Drinfeld center of a fusion 1-category of non-zero global dimension is a braided fusion 1-category. Further, it was later proven in \cite{ENO1} that every fusion 1-category over an algebraically closed field of characteristic zero has non-zero global dimension, so that the Drinfeld center of any such fusion 1-category is a braided fusion 1-category. In addition, it was established in \cite{O2} that the Drinfeld center of fusion 1-categories is invariant under Morita equivalence. This categorifies the well-known result that the center of an algebra is invariant under Morita equivalence. Furthermore, it was shown in \cite{ENO2} that the Drinfeld center of a fusion 1-category completely characterizes its Morita equivalence class.

The non-vanishing of the global dimension of fusion 1-categories is also a crucial property for the constructions of invariants of 3-manifolds. Namely, this hypothesis is needed in order to extend the Turaev-Viro construction to spherical fusion 1-categories \cite{BW}, and to extend the Reshetikhin-Turaev construction to modular 1-categories \cite{T1}. Fusion 1-categories are also related to topological quantum field theories through the cobordism hypothesis of \cite{BD} and \cite{L}, which asserts that framed fully extended TQFTs can be constructed by specifying a fully dualizable object in a symmetric monoidal higher category. Conceptually, full dualizability should be thought of as a very strong finiteness condition. In \cite{DSPS13}, the authors considered the symmetric monoidal 3-category $\mathbf{TC}$ of finite tensor 1-categories, and exhibited fully dualizable objects, i.e. 3-dualizable objects, therein. More precisely, they showed that a fusion 1-category $\mathcal{C}$ is a fully dualizable object of $\mathbf{TC}$ if and only if it is separable, meaning that $\mathcal{C}$ as a $\mathcal{C}$-$\mathcal{C}$-bimodule 1-category is equivalent to the 1-category of modules over a separable algebra in $\mathcal{C}\boxtimes\mathcal{C}^{mop}$. In addition, it was shown in \cite{DSPS13} that a fusion 1-category is separable if and only if its Drinfeld center is finite semisimple. But, over an algebraically closed field of characteristic zero, this last condition is always satisfied as we have recalled above, so that every fusion 1-category over such a field is separable, and therefore yields a framed fully extended 3-dimensional TQFT. Moreover, two fusion 1-categories are equivalent as objects of $\mathbf{TC}$ if and only if they are Morita equivalent. In particular, Morita equivalent fusion 1-categories yield equivalent 3-dimensional TQFTs.

Fusion 2-categories were introduced in \cite{DR} as a categorification of the notion of a fusion 1-category over an algebraically closed field of characteristic zero, and spherical fusion 2-categories were used to define a state-sum invariant of 4-manifolds. Further, Douglas and Reutter conjectured that fusion 2-categories are fully dualizable objects, that is 4-dualizable objects, in an appropriate symmetric monoidal 4-category. Additionally, we expect that the 4-dimensional TQFT constructed from a fusion 2-category using the cobordism hypothesis depends only on its Morita equivalence class as defined in \cite{D8}. Using the theory of higher condensations introduced in \cite{GJF}, and under some assumptions on the behaviour of higher colimits, the construction of the appropriate symmetric monoidal 4-category was given in \cite{JF}. Moreover, Johnson-Freyd sketched a characterization of the fully dualizable objects of this 4-category. In a different context, it was proven in \cite{BJS} that every braided fusion 1-category is a fully dualizable object of the symmetric monoidal 4-category $\mathbf{BrFus}$ of braided fusion 1-categories. This is related to the theory of fusion 2-categories. Namely, to any braided fusion 1-category $\mathcal{B}$, one can associate the fusion 2-category $\mathbf{Mod}(\mathcal{B})$ of finite semisimple $\mathcal{B}$-module 1-categories. Such fusion 2-categories are called connected, and it was shown in \cite{JFR} that connected fusion 2-categories have finite semisimple Drinfeld centers.

In the present article, we prove that every fusion 2-category is separable using the previous work of the author on fusion 2-categories in \cite{D3}, \cite{D4}, \cite{D7}, and \cite{D8}. As a byproduct, we show that, over an algebraically closed field of characteristic zero, every fusion 2-category is Morita equivalent to a connected fusion 2-category. Moreover, we prove that the Drinfeld center of any fusion 2-category is a finite semisimple 2-category.

\subsection*{Drinfeld Centers of Fusion 1-Categories}
\addcontentsline{toc}{subsection}{Drinfeld Centers of Fusion 1-Categories}

We now recall in more detail some of the relevant properties of fusion 1-categories. Given any fusion 1-category $\mathcal{C}$, we can canonically view $\mathcal{C}$ as a left $\mathcal{C}\boxtimes\mathcal{C}^{mop}$-module 1-category. Thanks to a result of Ostrik \cite{O1}, there exists an algebra $A$ in $\mathcal{C}\boxtimes\mathcal{C}^{mop}$ such that $\mathcal{C}$ as a left $\mathcal{C}\boxtimes\mathcal{C}^{mop}$-module 1-category is equivalent to the 1-category of right $A$-modules in $\mathcal{C}\boxtimes\mathcal{C}^{mop}$. We say that $\mathcal{C}$ is a separable fusion 1-category if $A$ is a separable algebra. In fact, $A$ admits a canonical Frobenius algebra structure, and $A$ is a special Frobenius algebra if and only if it is a separable algebra. It was established in \cite{DSPS13} that $A$ is special if and only if the global dimension of $\mathcal{C}$ is non-zero. But, over an algebraically closed field of characteristic zero, it was proven in \cite{ENO1} that every fusion 1-category has non-vanishing global dimension, so that every fusion 1-category over such a field is separable. Further, this implies that, over an algebraically closed field of characteristic zero, the Drinfeld center of any fusion 1-category is a finite semisimple 1-category. To see this, recall from \cite{O2} that the Drinfeld center $\mathcal{Z}(\mathcal{C})$ is equivalent to the tensor 1-category of $\mathcal{C}\boxtimes\mathcal{C}^{mop}$-module endofunctors on $\mathcal{C}$. As explained in \cite{O1}, this implies that $\mathcal{Z}(\mathcal{C})$ is equivalent to the 1-category of $A$-$A$-bimodules in $\mathcal{C}\boxtimes\mathcal{C}^{mop}$. Now, the algebra $A$ is separable if and only if the associated 1-category of $A$-$A$-bimodules in $\mathcal{C}\boxtimes\mathcal{C}^{mop}$ is finite semisimple. But, we have seen above that the algebra $A$ is separable for any fusion 1-category $\mathcal{C}$ over an algebraically closed field of characteristic zero, so that $\mathcal{Z}(\mathcal{C})$ is indeed a finite semisimple 1-category. Another key property of the Drinfeld center is that it is invariant under Morita equivalence of fusion 1-categories \cite{O2}. This result was refined further in \cite{ENO3}, where it is shown that two fusion 1-categories are Morita equivalent if and only if their Drinfeld centers are equivalent as braided fusion 1-categories.

\subsection*{Drinfeld Centers and Morita Equivalence Classes of Fusion 2-Categories}
\addcontentsline{toc}{subsection}{Drinfeld Centers and Morita Equivalence Classes of Fusion 2-Categories}

Let $\mathfrak{C}$ be a multifusion 2-category over an algebraically closed field of characteristic zero. The Drinfeld center $\mathscr{Z}(\mathfrak{C})$ of $\mathfrak{C}$ was defined in \cite{BN}. The objects of $\mathscr{Z}(\mathfrak{C})$ are given by objects of $\mathfrak{C}$ equipped with a half braiding and coherence data. In particular, $\mathscr{Z}(\mathfrak{C})$ is canonically a braided monoidal 2-category. We establish our first theorem using the notion of Morita equivalence between multifusion 2-categories developed in \cite{D8}.

\begin{namedthm*}{Theorem \ref{thm:MoritaInvarianceDrinfeldCenter}}
Let $\mathfrak{C}$ and $\mathfrak{D}$ be Morita equivalent multifusion 2-categories, then there is an equivalence $\mathscr{Z}(\mathfrak{C})\simeq\mathscr{Z}(\mathfrak{D})$ of braided monoidal 2-categories.
\end{namedthm*}

\noindent We note that, unlike in the case of fusion 1-categories, Morita equivalence classes of fusion 2-categories are not completely characterized by their Drinfeld centers. This failures can be attributed to the existence of non-trivial invertible fusion 2-categories, that is fusion 2-categories whose Drinfeld center is $\mathbf{2Vect}$. For instance, given any non-degenerate braided fusion 1-category $\mathcal{B}$, the fusion 2-category $\mathbf{Mod}(\mathcal{B})$ is invertible.

We then investigate the notion of Morita equivalence between connected fusion 2-categories, that is fusion 2-categories of the form $\mathbf{Mod}(\mathcal{B})$ with $\mathcal{B}$ a braided fusion 1-category. Specifically, for any symmetric fusion 1-category $\mathcal{E}$, we introduce a notion of Witt equivalence between braided fusion 1-categories with symmetric center $\mathcal{E}$, which we compare with the concept of Witt equivalence over $\mathcal{E}$ studied in \cite{DNO}. We obtain a 2-categorical characterization of our notion of Witt equivalence.

\begin{namedthm*}{Theorem \ref{thm:WittEquivalence}}
Let $\mathcal{B}_1$ and $\mathcal{B}_2$ be two braided fusion 1-categories. The fusion 2-categories $\mathbf{Mod}(\mathcal{B}_1)$ and $\mathbf{Mod}(\mathcal{B}_2)$ are Morita equivalent if and only if $\mathcal{B}_1$ and $\mathcal{B}_2$ have the same symmetric center $\mathcal{E}$ and they are Witt equivalent.
\end{namedthm*}

\noindent This theorem generalizes example 5.4.6 of \cite{D8}, where it is shown that if $\mathcal{B}_1$ and $\mathcal{B}_2$ are non-degenerate braided fusion 1-categories, then $\mathbf{Mod}(\mathcal{B}_1)$ and $\mathbf{Mod}(\mathcal{B}_2)$ are Morita equivalent fusion 2-categories if and only if $\mathcal{B}_1$ and $\mathcal{B}_2$ are Witt equivalent in the sense of \cite{DMNO}. We use the above result to classify Morita autoequivalences of connected symmetric fusion 2-categories. In addition, we explain how the notion of Witt equivalence over $\mathcal{E}$ from \cite{DNO} is recovered by that of Morita equivalence between fusion 2-categories equipped with a central functor from $\mathbf{Mod}(\mathcal{E})$.

We go on to study fusion 2-categories up to Morita equivalence. Recall from \cite{JFY} that a strongly fusion 2-category is fusion 2-category whose braided fusion 1-category of endomorphisms of the monoidal unit is the symmetric monoidal 1-category of (super) vector spaces. Using this notion, we establish the following theorem, whose proof is inspired by the classification of topological orders studied in \cite{JF} and relies on the minimal non-degenerate extension conjecture established in \cite{JFR}.

\begin{namedthm*}{Theorem \ref{thm:stronglyfusioninvertible}}
Every fusion 2-category is Morita equivalent to the 2-Deligne tensor product of a strongly fusion 2-category and an invertible fusion 2-category.
\end{namedthm*}

\noindent We note that the invertible fusion 2-category supplied by the above theorem is in general not unique. The strongly fusion 2-category is not uniquely determined either, as will be explained in \cite{JF3}. We then prove that every strongly fusion 2-category is Morita equivalent to a connected fusion 2-category, which yields the following theorem.

\begin{namedthm*}{Theorem \ref{thm:Moritaconnected}}
Every fusion 2-category is Morita equivalent to a connected fusion 2-category.
\end{namedthm*}

\noindent In particular, we obtain a complete classification of the Morita equivalence classes of fusion 2-categories in terms of quotients of the Witt groups $\mathcal{W}(\mathcal{E})$ of braided fusion 1-categories over the symmetric fusion 1-category $\mathcal{E}$ introduced in \cite{DNO}. Moreover, if $\mathcal{E}$ is Tannakian, we subsequently describe $\mathcal{W}(\mathcal{E})$ as a set.

An algebra in a fusion 2-category is called rigid if its multiplication 1-morphism has a right adjoint as a 1-morphism of bimodules. Rigid algebras in $\mathbf{2Vect}$ are exactly multifusion 1-categories. We therefore view the concept of a rigid algebra in a fusion 2-category as a generalization, or more precisely an internalization, of the definition of a multifusion 1-category. Further, the property of being separable for multifusion 1-categories can be generalized to a notion of separability for rigid algebras in fusion 2-categories. Thanks to our last theorem, we deduce the result below, which internalizes corollary 2.6.8 of \cite{DSPS13}, and affirmatively answers a question raised in \cite{JFR}.

\begin{namedthm*}{Theorem \ref{thm:algebrarigidseparable}}
Every rigid algebra in a fusion 2-category is separable.
\end{namedthm*}

\noindent We also explain how this result implies that every rigid algebra in a multifusion 2-category is separable.

We move on to study the separability of multifusion 2-categories. More precisely, any multifusion 2-category $\mathfrak{C}$ may be viewed as a left $\mathfrak{C}\boxtimes \mathfrak{C}^{mop}$-module 2-category. By the main theorem of \cite{D4}, there exists a canonical rigid algebra $\mathcal{R}_{\mathfrak{C}}$ in $\mathfrak{C}\boxtimes \mathfrak{C}^{mop}$ such that $\mathfrak{C}$ is equivalent to the 2-category of right $\mathcal{R}_{\mathfrak{C}}$-modules in $\mathfrak{C}\boxtimes \mathfrak{C}^{mop}$. We say that $\mathfrak{C}$ is separable if $\mathcal{R}_{\mathfrak{C}}$ is a separable algebra. Thanks to our last theorem above, we obtain the following categorification of corollary 2.6.8 of \cite{DSPS13}.

\begin{namedthm*}{Corollary \ref{cor:separable}}
Every multifusion 2-category is separable.
\end{namedthm*}

\noindent Further, by relying on the sketch given in \cite{JF}, we outline an argument showing that separability is equivalent to 4-dualizability for multifusion 2-categories. 

In case $\mathfrak{C}$ is a fusion 2-category, the rigid algebra $\mathcal{R}_{\mathfrak{C}}$ is connected in the sense that its unit 1-morphism is simple. This allows us to use the construction of section 3.2 of \cite{D7} to associate to $\mathcal{R}_{\mathfrak{C}}$ a scalar, called the dimension. Accordingly, we define the dimension $\mathrm{Dim}(\mathfrak{C})$ of the fusion 2-category $\mathfrak{C}$ to be the dimension of the canonical rigid algebra $\mathcal{R}_{\mathfrak{C}}$. Using this notion, we obtain a categorification of theorem 2.3 of \cite{ENO1}.

\begin{namedthm*}{Corollary \ref{cor:nonzerodimension}}
The dimension $\mathrm{Dim}(\mathfrak{C})$ of any fusion 2-category is non-zero.
\end{namedthm*}

\noindent Finally, we explain how to derive the following corollary from our previous theorems.

\begin{namedthm*}{Corollary \ref{cor:F2Ccenter}}
The Drinfeld center of any multifusion 2-category is a finite semisimple 2-category.
\end{namedthm*}

\noindent These last results have many immediate consequences. In particular, they imply that the Drinfeld center commutes with the 2-Deligne tensor product for multifusion 2-categories. We also find that the Drinfeld center of any fusion 2-category is equivalent as a braided fusion 2-category to the Drinfeld center of a strongly fusion 2-category.

\subsection*{Outline}
\addcontentsline{toc}{subsection}{Outline}

In section \ref{sec:preliminaries}, we review the definitions of a finite semisimple 2-category and of a fusion 2-category. We survey some examples of fusion 2-categories including connected fusion 2-categories and strongly fusion 2-categories. In addition, we recall the construction of the 2-Deligne tensor product, as well as the definition of a separable algebra in a fusion 2-category. We also review the Morita theory of fusion 2-categories.

Next, in section \ref{sec:center}, we begin by recalling the definition of the Drinfeld center of a monoidal 2-category, and go on to uncover its elementary properties. For instance, we show that the Drinfeld center of a rigid monoidal 2-category is a rigid monoidal 2-category. We go on to study more specifically the properties of the Drinfeld centers of multifusion 2-categories over an algebraically closed field of characteristic zero. In particular, we show that the Drinfeld center is invariant under Morita equivalence of fusion 2-categories.

Then, in section \ref{sec:connected}, we examine in detail the Morita theory of connected fusion 2-categories. We establish that the notion of Morita equivalence between connected fusion 2-categories corresponds to a notion of Witt equivalence between braided fusion 1-categories. As an application, we classify Morita autoequivalences of connected symmetric fusion 2-categories.

In section \ref{sec:F2CmoduloMorita}, we show that every fusion 2-category is Morita equivalent to the 2-Deligne tensor product of a strongly fusion 2-category and an invertible fusion 2-category. Subsequently, we prove that every fusion 2-category is Morita equivalent to a connected fusion 2-category by showing that every strongly fusion 2-category is Morita equivalent to a connected fusion 2-category.

Finally, in section \ref{sec:separable}, we prove that every rigid algebra in a fusion 2-category is separable. As a consequence, we show that every fusion 2-category is separable. Then, we sketch an argument showing that separability is equivalent to 4-dualizability for fusion 2-categories. We continue by defining the dimension of a fusion 2-category, which we argue is always non-vanishing. We then show that the Drinfeld center of any fusion 2-category is a finite semisimple 2-category, and that taking Drinfeld centers commutes with the 2-Deligne tensor product.

\subsection*{Acknowledgments}

I would like to express my gratitude towards Christopher Douglas and David Reutter for their extensive feedback on an earlier version of this manuscript.

\input{Preliminaries}

\input{Center}

\input{Connected}

\input{F2CMorita}

\input{Separable}

\bibliography{bibliography.bib}

\end{document}

%% file: Preliminaries.tex
\section{Preliminaries}\label{sec:preliminaries}

\subsection{Finite Semisimple 2-Categories}\label{sub:semisimple2categories}

We begin by reviewing the definition of a 2-condensation monad introduced in \cite{GJF} as a categorification of the notions of an idempotent. More precisely, we recall the unpacked version of this definition given in section 1.1 of \cite{D1}.

\begin{Definition}
A 2-condensation monad in a 2-category $\mathfrak{C}$ is an object $A$ of $\mathfrak{C}$ equipped with a 1-morphism $e:A\rightarrow A$ and two 2-morphisms $\mu:e\circ e\Rightarrow e$ and $\delta:e\circ e\Rightarrow e$ such that $\mu$ is associative, $\delta$ is coassociative, the Frobenius relations hold, i.e. $\delta$ is a 2-morphism of $e$-$e$-bimodules, and $\mu\cdot \delta=Id_e$.
\end{Definition}

\noindent Categorifying the notion of split surjection, \cite{GJF} gave the definition of a 2-condensation, which we recall below.

\begin{Definition}
A 2-condensation in a 2-category $\mathfrak{C}$ is a pair of objects $A,B$ in $\mathfrak{C}$ together with two 1-morphisms $f:A\rightarrow B$ and $g:B\rightarrow A$ and two 2-morphisms $\phi:f\circ g\Rightarrow Id_B$ and $\gamma:Id_B\Rightarrow f\circ g$ such that $\phi\cdot\gamma=Id_{Id_B}$.
\end{Definition}

\noindent We say that a 2-condensation monad splits if it can be extended to a 2-condensation (see definition 1.1.4 of \cite{D1} for details).

Let us now fix a field $\mathds{k}$. Following \cite{GJF}, we will call a $\mathds{k}$-linear 2-category locally Cauchy complete if its $Hom$-categories are Cauchy complete, that is they have direct sums and idempotents split. Furthermore, a Cauchy complete linear 2-category is a locally Cauchy complete 2-category that has direct sums for objects and such that 2-condensation monads split. Given an arbitrary locally Cauchy complete 2-category $\mathfrak{C}$, one can form its Cauchy completion $Cau(\mathfrak{C})$, which is Cauchy complete (see \cite{GJF}, and also \cite{DR} for a slightly different approach). Further, it was shown in proposition 1.2.5 of \cite{D1} that $Cau(\mathfrak{C})$ is 3-universal with respect to linear 2-functors from $\mathfrak{C}$ to a Cauchy complete linear 2-category. Let us also mention that, given two Cauchy complete linear 2-category $\mathfrak{C}$ and $\mathfrak{D}$, one can form their Cauchy completed tensor product $\mathfrak{C}\widehat{\otimes}\mathfrak{D}$, which is 3-universal with respect to bilinear 2-functor from $\mathfrak{C}\times\mathfrak{D}$ to a Cauchy complete linear 2-category (see proposition 3.4 of \cite{D3}).

From now on, we will assume that $\mathds{k}$ is an algebraically closed field of characteristic zero. We are ready to review the definitions of a semisimple 2-category and of a finite semisimple 2-category introduced in \cite{DR} (more precisely, we review the variants given in \cite{D1}).

\begin{Definition}
A linear 2-category is semisimple if it is locally semisimple, has right and left adjoints for 1-morphisms, and is Cauchy complete.
\end{Definition}

\noindent An object $C$ of a semisimple 2-category $\mathfrak{C}$ is called simple if the identity 1-morphism $Id_C$ is a simple object of the semisimple 1-category $End_{\mathfrak{C}}(C)$.

\begin{Definition}
A semisimple linear 2-category is finite if it is locally finite semisimple and it has finitely many equivalence classes of simple objects.
\end{Definition}

\noindent Given a multifusion 1-category $\mathcal{C}$, the 2-category $\mathbf{Mod}(\mathcal{C})$ of finite semisimple right $\mathcal{C}$-module 1-categories is a finite semisimple 2-category (see theorem 1.4.8 of \cite{DR}). Moreover, by theorem 1.4.9 of \cite{DR}, every finite semisimple 2-category is of this form. If $\mathfrak{C}$ and $\mathfrak{D}$ are finite semisimple, then it follows from theorem 3.7 of \cite{D3} that $\mathfrak{C}\widehat{\otimes}\mathfrak{D}$ is a finite semisimple 2-category, which we denote by $\mathfrak{C}\boxtimes\mathfrak{D}$, and call the 2-Deligne tensor product of $\mathfrak{C}$ and $\mathfrak{D}$.

Given two simple objects $C$ and $D$ of a finite semisimple 2-category $\mathfrak{C}$. The finite semisimple 1-category $Hom_{\mathfrak{C}}(C,D)$ may be non-zero. If this is the case, we say that $C$ and $D$ are connected. We emphasize that two non-equivalent simple objects may be connected. This is in sharp contrast with the decategorified setting, in which the Schur lemma prevents such behaviour. Further, it was shown in section 1.2.3 of \cite{DR} that being connected defines an equivalence relation on the set of simple objects.

\begin{Definition}
Let $\mathfrak{C}$ be a finite semisimple 2-category. We write $\pi_0(\mathfrak{C})$ for the quotient of the set of simple objects of $\mathfrak{C}$ under the relation of being connected.
\end{Definition}

\noindent We say that a finite semisimple 2-category $\mathfrak{C}$ is connected if $\pi_0(\mathfrak{C})$ is a singleton. For instance, if $\mathcal{C}$ is a fusion 1-category, then $\mathbf{Mod}(\mathcal{C})$ is connected. Further, we will abuse the terminology, and identify a connected component, i.e.\ an equivalence class in $\pi_0(\mathfrak{C})$, with the finite semisimple sub-2-category of $\mathfrak{C}$ that is spanned by the simple objects in this equivalence class and full on 1-morphisms.

\subsection{Fusion 2-Categories}\label{sub:exampleF2C}

We recall the definition of a fusion 2-category first introduced in \cite{DR} (more precisely, we review the variant considered in \cite{D2}), and give a number of key examples. Throughout, we work over $\mathds{k}$, an algebraically closed field of characteristic zero.

\begin{Definition}
A multifusion 2-category is a rigid monoidal finite semisimple 2-category. A fusion 2-category is multifusion 2-category whose monoidal unit is a simple object.
\end{Definition}

\begin{Example}
We write $\mathbf{2Vect}$ for the 2-category of finite semisimple ($\mathds{k}$-linear) 1-categories, also called 2-vector spaces. The Deligne tensor product endows $\mathbf{2Vect}$ with the structure of a fusion 2-category.
\end{Example}

\begin{Example}\label{ex:connectedfusion2categories}
Let $\mathcal{B}$ be a braided fusion 1-category, then the relative Deligne tensor product over $\mathcal{B}$ endows the 2-category $\mathbf{Mod}(\mathcal{B})$ with a rigid monoidal structure, so that $\mathbf{Mod}(\mathcal{B})$ is a connected fusion 2-category. In fact, every connected fusion 2-category is of this form (see section 2.4 of \cite{D2}). Further, connected fusion 2-categories are ubiquitous. Namely, if $\mathfrak{C}$ is an arbitrary fusion 2-category, then $\mathfrak{C}^0$, the connected component of the identity of $\mathfrak{C}$, is a connected fusion 2-category by proposition 2.4.5 of \cite{D2}. More precisely, if we let $\Omega\mathfrak{C}$ denote the braided fusion 1-category of endormophisms of the monoidal unit of $\mathfrak{C}$, then $\mathfrak{C}^0\simeq \mathbf{Mod}(\Omega\mathcal{B})$ as fusion 2-categories.
\end{Example}

\begin{Example}\label{ex:2representations}
Let $G$ be a finite group. We write $\mathbf{2Rep}(G)$ for the 2-category of finite semisimple 1-categories equipped with an action by $G$. The Deligne tensor product endows the finite semisimple 2-category $\mathbf{2Rep}(G)$ with a symmetric monoidal structure. In fact, it follows from lemma 1.3.8 of \cite{D8} that there is an equivalence of monoidal 2-categories $\mathbf{2Rep}(G)\simeq \mathbf{Mod}(\mathbf{Rep}(G))$, so that $\mathbf{2Rep}(G)$ is a fusion 2-category.
\end{Example}

The following classes of fusion 2-categories were first singled out in \cite{JFY}, and will play a key role in the proof of our main theorem.

\begin{Definition}\label{def:stronglyfusion}
A bosonic strongly fusion 2-category is fusion 2-category $\mathfrak{C}$ for which $\Omega\mathfrak{C}\simeq\mathbf{Vect}$ as braided fusion 1-categories. A fermionic strongly fusion 2-category is fusion 2-category $\mathfrak{C}$ for which $\Omega\mathfrak{C}\simeq\mathbf{SVect}$, the symmetric fusion 1-category of super vector spaces.
\end{Definition}

\begin{Example}\label{ex:twistedgroupgraded2vectorspaces}
Let $G$ be a finite group. We use $\mathbf{2Vect}_G$ to denote the finite semisimple 2-category of $G$-graded 2-vector spaces. The convolution product turns $\mathbf{2Vect}_G$ into a fusion 2-category. Furthermore, given a 4-cocycle $\pi$ for $G$ with coefficients in $\mathds{k}$, we can form the fusion 2-category $\mathbf{2Vect}^{\pi}_G$ by twisting the structure 2-isomorphisms of $\mathbf{2Vect}_G$ using $\pi$ (see construction 2.1.16 of \cite{DR}). In fact, it follows from theorem A of \cite{JFY} that every bosonic strongly fusion 2-category is of this form.
\end{Example}

\begin{Remark}
Let us write $\mathbf{2SVect}:=\mathbf{Mod}(\mathbf{SVect})$ for the fusion 2-category of super 2-vector spaces. Give a finite group $G$, we can consider the fusion 2-category $\mathbf{2SVect}_G$ of $G$-graded super 2-vector spaces. One can use a 4-cocycle for $G$ with coefficients in the (extended) supercohomology of \cite{GJF0} to obtain another fusion 2-category by twisting the coherence 2-isomorphisms of $\mathbf{2SVect}_G$. Such fusion 2-categories are examples of fermionic strongly fusion 2-categories. However, not every strongly fusion 2-category is of this form as can be seen from example 2.1.27 of \cite{DR}. Nonetheless, theorem B of \cite{JFY} shows that every object a fermionic strongly fusion 2-category is invertible, so that every connected component of such a fusion 2-category is equivalent to $\mathbf{2SVect}$.
\end{Remark}

\begin{Example}\label{ex:2groupgraded2vectorspaces}
Given a finite 2-group $\mathcal{G}$, we can consider the fusion 2-category $\mathbf{2Vect}_{\mathcal{G}}$ of $\mathcal{G}$-graded 2-vector spaces as in construction 2.1.13 of \cite{DR}. Succinctly, $\mathbf{2Vect}_{\mathcal{G}}$ is constructed as follows. We define a monoidal 2-category by $\mathfrak{X}:=\mathcal{G}\times \mathrm{B}^2\mathds{k}$, the free $\mathds{k}$-linear monoidal 2-category on $\mathcal{G}$. We then write $\widetilde{\mathfrak{X}}$ for its local Cauchy completion, and finally set $\mathbf{2Vect}_{\mathcal{G}}:=Cau(\widetilde{\mathfrak{X}})$. More generally, given a 4-cocycle $\pi$ for $\mathcal{G}$ with coefficients in $\mathds{k}^{\times}$, we can consider the fusion 2-category $\mathbf{2Vect}^{\pi}_{\mathcal{G}}$ of $\pi$-twisted $\mathcal{G}$-graded 2-vector spaces (see construction 2.1.16 of \cite{DR}). Namely, the 4-cocycle $\pi$ provides us with the Postnikov data necessary to form the 3-group $\mathcal{G}\times^{\pi} \mathrm{B}^2\mathds{k}^{\times}$, which we think of as a non-linear monoidal 2-category. After linearization, we obtain a linear monoidal 2-category $\mathfrak{Y}:=\mathcal{G}\times^{\pi} \mathrm{B}^2\mathds{k}$, which is a twisted variant of $\mathfrak{X}$ above. We then set $\mathbf{2Vect}^{\pi}_{\mathcal{G}}:=Cau(\widetilde{\mathfrak{Y}})$, with $\widetilde{\mathfrak{Y}}$ the local Cauchy completion of $\mathfrak{Y}$.
\end{Example}

\begin{Example}
Let $G$ be a finite group, and $\pi$ a 4-cocycle for $G$ with coefficients in $\mathds{k}^{\times}$. We may consider the braided monoidal 2-category $\mathscr{Z}(\mathbf{2Vect}_G^{\pi})$, the Drinfeld center of the fusion 2-category $\mathbf{2Vect}_G^{\pi}$. We review this definition in section \ref{sub:definitioncenter} below. It was shown directly in \cite{KTZ} that $\mathscr{Z}(\mathbf{2Vect}_G^{\pi})$ is a finite semisimple 2-category. In fact, $\mathscr{Z}(\mathbf{2Vect}_G^{\pi})$ is rigid as we will show in corollary \ref{cor:Drinfeldrigid}, so that it really is a fusion 2-category.
\end{Example}

\begin{Example}
Let $\mathfrak{C}$ and $\mathfrak{D}$ be two fusion 2-categories. It follows from theorem 5.6 and lemma 5.7 of \cite{D3} that their 2-Deligne tensor product, which we denote by $\mathfrak{C}\boxtimes\mathfrak{D}$, is a fusion 2-category. In addition, it follows from lemma 4.3 of \cite{D3} that $(\mathfrak{C}\boxtimes\mathfrak{D})^0\simeq \mathfrak{C}^0\boxtimes\mathfrak{D}^0$. Further, by proposition 4.1 of \cite{D3}, we have that $\Omega (\mathfrak{C}\boxtimes\mathfrak{D})\simeq \Omega\mathfrak{C}\boxtimes \Omega\mathfrak{D}$ as braided fusion 1-categories.
\end{Example}

\subsection{Rigid and Separable Algebras}

Let us fix $\mathfrak{C}$ a fusion 2-category (over an algebraically closed field of characteristic zero), with monoidal product $\Box$ and monoidal unit $I$. An algebra, also called pseudo-monoid in \cite{DS}, in $\mathfrak{C}$ is an object $A$ of $\mathfrak{C}$ equipped with a multiplication 1-morphism $m:A\Box A\rightarrow A$, a unit 1-morphism $i:I\rightarrow A$, and coherence 2-isomorphisms ensuring that $m$ is associative and unital. We focus our attention on the following refinement of the notion of an algebra, which was introduced in \cite{G}, and was first considered in the context of fusion 2-categories in \cite{JFR}.

\begin{Definition}
A rigid algebra in $\mathfrak{C}$ is an algebra $A$ whose multiplication 1-morphism $m:A\Box A\rightarrow A$ has a right adjoint $m^*$ as an $A$-$A$-bimodule 1-morphism.
\end{Definition}

\noindent For our purposes, it is necessary to single out the following class of rigid algebras introduced in \cite{JFR}, and whose origin can be traced back to \cite{GJF}.

\begin{Definition}
A separable algebra in $\mathfrak{C}$ is a rigid algebra $A$ in $\mathfrak{C}$ equipped with a section of the counit $m\circ m^*\Rightarrow Id_A$ as an $A$-$A$-bimodule 2-morphism.
\end{Definition}

As we will briefly recall below, separable algebras can be used to define the notion of Morita equivalence between fusion 2-categories. This is thanks to the following result, which is a combination of theorem 3.1.6 of \cite{D7} and corollary 2.2.3 of \cite{D5}.

\begin{Theorem}\label{thm:bimodulefinitesemisimple}
Let $A$ be a rigid algebra in a fusion 2-category $\mathfrak{C}$. Then, $A$ is separable if and only if $\mathbf{Bimod}_{\mathfrak{C}}(A)$, the 2-category of $A$-$A$-bimodules in $\mathfrak{C}$, is finite semisimple. Further, if either of these conditions is satisfied, $\mathbf{Mod}_{\mathfrak{C}}(A)$, the 2-category of right $A$-modules in $\mathfrak{C}$, is finite semisimple.
\end{Theorem}

\noindent It is in general difficult to prove that a rigid algebra is separable. Nonetheless, if $A$ is a connected rigid algebra, meaning that $i:I\rightarrow A$ is a simple 1-morphism, then one can associate to $A$ a scalar $\mathrm{Dim}_{\mathfrak{C}}(A)$ (see section 3.2 of \cite{D7}). It was established in theorem 3.2.4 of \cite{D7} that $A$ is separable if and only if $\mathrm{Dim}_{\mathfrak{C}}(A)$ is non-zero. This criterion can be effectively used to give examples of separable algebras.

\begin{Example}
Rigid algebras in $\mathbf{2Vect}$ are precisely multifusion 1-categories. It follows from corollary 2.6.8 of \cite{DSPS13} that every such rigid algebra is separable. More generally, if $G$ is a finite group, rigid algebras in $\mathbf{2Vect}_G$ are exactly $G$-graded multifusion 1-categories. It was argued in example 5.2.4 of \cite{D4} that every such rigid algebra is separable.
\end{Example}

\begin{Example}\label{ex:algebrasModB}
Let $\mathcal{B}$ be a braided fusion 1-category. In the terminology of \cite{BJS}, a $\mathcal{B}$-central monoidal 1-category is a monoidal linear 1-category $\mathcal{C}$ equipped with a braided monoidal functor $F:\mathcal{B}\rightarrow \mathcal{Z}(\mathcal{C})$ to the Drinfeld center of $\mathcal{C}$. By lemma 2.1.4 of \cite{D7}, rigid algebras in $\mathbf{Mod}(\mathcal{B})$ are exactly $\mathcal{B}$-central multifusion 1-categories. By combining proposition 3.3.3 of \cite{D7} together with theorem 2.3 of \cite{ENO1}, we find that every $\mathcal{B}$-central fusion 1-category is separable. More generally, every $\mathcal{B}$-central multifusion 1-category is separable as we explain below.

Namely, let $\mathcal{C}$ be a $\mathcal{B}$-central multifusion 1-category. We use $\mathcal{C}^{mop}$ to denote the multifusion 1-category obtained from $\mathcal{C}$ by reversing the order of the monoidal product. Note that there is a canonical $\mathcal{B}$-central structure on $\mathcal{C}^{mop}$. Then, inspection shows that there is an equivalence $$\mathbf{Bimod}_{\mathbf{Mod}(\mathcal{B})}(\mathcal{C})\simeq \mathbf{Mod}(\mathcal{C}^{mop}\boxtimes_{\mathcal{B}}\mathcal{C})$$ of 2-categories. Moreover, it follows from theorem 6.2 of \cite{Gre} that $\mathcal{C}^{mop}\boxtimes_{\mathcal{B}}\mathcal{C}$ is a rigid monoidal 1-category. Further, it is finite semisimple by proposition 3.5 of \cite{ENO2}. This shows that $\mathcal{C}$ is indeed separable.
\end{Example}

\begin{Remark}\label{rem:balancedbimdoule1categories}
The 2-category $\mathbf{Bimod}_{\mathbf{Mod}(\mathcal{B})}(\mathcal{C})$ is the 2-category of finite semisimple $\mathcal{B}$-balanced $\mathcal{C}$-$\mathcal{C}$-bimodule 1-categories in the sense of definition 3.24 \cite{Lau}. Namely, the author assumes that $\mathcal{C}$ is a $\mathcal{B}$-augmented monoidal 1-category as in definition 3.12 therein, but only the underlying $\mathcal{B}$-central structure on $\mathcal{C}$ is relevant for the definition of $\mathcal{B}$-balanced $\mathcal{C}$-$\mathcal{C}$-bimodule 1-categories. Similarly, let us also point out that the equivalence of 2-categories given above in example \ref{ex:algebrasModB} is essentially theorem 3.27 of \cite{Lau}.
\end{Remark}

\begin{Example}\label{ex:algebrasZ2VG}
Let $G$ be a finite group. Inspection shows that rigid algebras in $\mathscr{Z}(\mathbf{2Vect}_G)$ are precisely $G$-crossed multifusion 1-categories in the sense of \cite{Gal} (see \cite{T2} for the original definition). By corollary 3.3.6 of \cite{D7}, every $G$-crossed fusion 1-categories is a separable algebra in $\mathscr{Z}(\mathbf{2Vect}_G)$. Moreover, braided rigid algebras in $\mathscr{Z}(\mathbf{2Vect}_G)$ are exactly $G$-crossed braided multifusion 1-categories in the sense of \cite{Gal} (see also \cite{T2} for the first definition).
\end{Example}

\subsection{The Morita Theory of Fusion 2-Categories}

We end this section be reviewing the notion of Morita equivalence between fusion 2-categories introduced in \cite{D8}. Given a fusion 2-category $\mathfrak{C}$, a left $\mathfrak{C}$-module 2-category is a 2-category $\mathfrak{M}$ equipped with a coherent left action by $\mathfrak{C}$ (see definition 2.1.3 of \cite{D4} for details). If $\mathfrak{M}$ is a finite semisimple 2-category, then it follows from theorem 5.3.4 of \cite{D4} that there exists a rigid algebra $A$ in $\mathfrak{C}$, together with an equivalence $\mathfrak{M}\simeq \mathbf{Mod}_{\mathfrak{C}}(A)$ of left $\mathfrak{C}$-module 2-categories. We say that $\mathfrak{M}$ is separable if $A$ is separable, and note that this property does not depend on the algebra $A$ (see proposition 5.1.3 of \cite{D8}).

Given a separable left $\mathfrak{C}$-module 2-category $\mathfrak{M}$, we use $\mathfrak{C}^*_{\mathfrak{M}}$ to denote $\mathbf{End}_{\mathfrak{C}}(\mathfrak{M})$, the monoidal 2-category of left $\mathfrak{C}$-module 2-endofunctors on $\mathfrak{M}$ (see theorem 4.1.7 of \cite{D8}), and call it the dual tensor 2-category to $\mathfrak{C}$ with respect to $\mathfrak{M}$. On the other hand, given a separable algebra $A$ in $\mathfrak{C}$, it follows from theorem 3.2.8 of \cite{D8} that the finite semisimple 2-category $\mathbf{Bimod}_{\mathfrak{C}}(A)$ has a canonical monoidal structure given by the relative tensor product over $A$. The next result is given by combining theorem 5.3.2 of \cite{D8} with corollary 2.2.3 of \cite{D5}.

\begin{Theorem}\label{thm:dualfusion2category}
Let $A$ be a separable algebra in a fusion 2-category $\mathfrak{C}$. Then, $$\mathbf{End}_{\mathfrak{C}}(\mathbf{Mod}_{\mathfrak{C}}(A))\simeq \mathbf{Bimod}_{\mathfrak{C}}(A)^{mop}$$ is a multifusion 2-category.
\end{Theorem}

\noindent In order to ensure that the multifusion 2-category of theorem \ref{thm:dualfusion2category} is fusion, we need to impose a condition on the algebra $A$. More precisely, we say that an algebra is indecomposable if it can not be written as the direct sum of two non-trivial algebras. For instance, any connected algebra is indecomposable. We use an analogous terminology for module 2-categories. The next result is obtained by combing theorem 5.4.3, which characterizes Morita equivalence between multifusion 2-categories, and section 5.2 of \cite{D8}.

\begin{Theorem}\label{thm:MoritaF2C}
For any two fusion 2-categories $\mathfrak{C}$ and $\mathfrak{D}$, the following are equivalent:
\begin{enumerate}
    \item There exists an indecomposable separable left $\mathfrak{C}$-module 2-category $\mathfrak{M}$, and an equivalence of monoidal 2-categories $\mathfrak{D}^{mop}\simeq \mathfrak{C}_{\mathfrak{M}}^*$.
    \item There exists an indecomposable separable algebra $A$ in $\mathfrak{C}$, and an equivalence of monoidal 2-categories $\mathfrak{D}\simeq \mathbf{Bimod}_{\mathfrak{C}}(A)$.
\end{enumerate}
If either of the above conditions is satisfied, we say that $\mathfrak{C}$ and $\mathfrak{D}$ are Morita equivalent. Furthermore, Morita equivalence defines an equivalence relation between fusion 2-categories.
\end{Theorem}

For later use, we also record the following technical lemmas.

\begin{Lemma}\label{lem:module2Deligne}
Let $\mathfrak{C}$ and $\mathfrak{D}$ be two multifusion 2-categories. For any separable algebras $A$ in $\mathfrak{C}$ and $B$ in $\mathfrak{D}$, there is an equivalence $$\mathbf{Mod}_{\mathfrak{C}}(A)\boxtimes\mathbf{Mod}_{\mathfrak{D}}(B)\simeq \mathbf{Mod}_{\mathfrak{C}\boxtimes\mathfrak{D}}(A\boxtimes B)$$ of left $\mathfrak{C}\boxtimes\mathfrak{D}$-module 2-categories.
\end{Lemma}
\begin{proof}
Thanks to the 3-universal property of the 2-Deligne tensor product obtained in theorem 3.7 of \cite{D3}, the bilinear 2-functor $$\begin{tabular}{ccc}$\mathbf{Mod}_{\mathfrak{C}}(A)\times\mathbf{Mod}_{\mathfrak{D}}(B)$& $\rightarrow$ & $\mathbf{Mod}_{\mathfrak{C}\boxtimes\mathfrak{D}}(A\boxtimes B).$\\ $(M,N)$&$\mapsto$ & $M\boxtimes N$\end{tabular}$$ induces a linear 2-functor $\mathbf{K}:\mathbf{Mod}_{\mathfrak{C}}(A)\boxtimes\mathbf{Mod}_{\mathfrak{D}}(B)\rightarrow \mathbf{Mod}_{\mathfrak{C}\boxtimes\mathfrak{D}}(A\boxtimes B).$ Furthermore, this 2-functor is compatible with the left $\mathfrak{C}\boxtimes\mathfrak{D}$-module structures. Now, it follows from the proof of proposition 3.1.2 of \cite{D7} that $\mathbf{Mod}_{\mathfrak{C}\boxtimes\mathfrak{D}}(A\boxtimes B)$ is generated under Cauchy completion by the full sub-2-category on the objects of the form $(C\Box D)\boxtimes (A\Box B)$ with $C$ in $\mathfrak{C}$ and $D$ in $\mathfrak{D}$. Likewise, it follows from its construction that $\mathbf{Mod}_{\mathfrak{C}}(A)\boxtimes\mathbf{Mod}_{\mathfrak{D}}(B)$ is generated under Cauchy completion by the full sub-2-category on the objects of the form $(C\Box A)\boxtimes (D\Box B)$ for every $C$ in $\mathfrak{C}$ and $D$ in $\mathfrak{D}$. But, we have that $\mathbf{K}\big((C\Box A)\boxtimes (D\Box B)\big)=(C\Box D)\boxtimes (A\Box B)$. Furthermore, it follows by combining proposition 4.1 of \cite{D3} and lemma 3.2.13 of \cite{D4} that $\mathbf{K}$ induces an equivalence between these two full sub-2-categories. This establishes that $\mathbf{K}$ is an equivalence as desired.
\end{proof}

\begin{Lemma}\label{lem:bimodule2Deligne}
Let $\mathfrak{C}$ and $\mathfrak{D}$ be two multifusion 2-categories. For any separable algebras $A$ in $\mathfrak{C}$ and $B$ in $\mathfrak{D}$, there is an equivalence $$\mathbf{Bimod}_{\mathfrak{C}}(A)\boxtimes\mathbf{Bimod}_{\mathfrak{D}}(B)\simeq \mathbf{Bimod}_{\mathfrak{C}\boxtimes\mathfrak{D}}(A\boxtimes B)$$ of multifusion 2-categories.
\end{Lemma}
\begin{proof}
The proof is analogous to that of the previous lemma.
\end{proof}

\begin{Remark}
In particular, the 2-Deligne tensor product is compatible with Morita equivalences of multifusion 2-categories.
\end{Remark}

%% file: Center.tex
\section{The Drinfeld Center}\label{sec:center}

In this section, we begin be reviewing the definition of the Drinfeld center of a monoidal 2-category. We go on to prove some of its elementary properties. Finally, we show that the Drinfeld center of a multifusion 2-category is invariant under Morita equivalence.

\subsection{Definition}\label{sub:definitioncenter}

Let $\mathfrak{C}$ be a monoidal 2-category in the sense of \cite{SP} with monoidal product $\Box$ and monoidal unit $I$. Thanks to \cite{Gur}, we may assume without loss of generality that $\mathfrak{C}$ is strict cubical in that it satisfies definition 2.26 of \cite{SP}. This ensures that almost all of the coherence data for $\mathfrak{C}$ is given by identities, so that we will often omit the symbols $\Box$ and $I$. The only exception being the interchanger 2-isomorphism $\phi^{\Box}$ witnessing that the 2-functor $\Box$ respects the composition of 1-morphisms. We now recall from \cite{Cr} the definition of the Drinfeld center $\mathscr{Z}(\mathfrak{C})$ of $\mathfrak{C}$, which is a braided monoidal 2-category (see also \cite{BN}). We do so using the variant of the graphical calculus of \cite{GS} introduced in \cite{D4}. For a version the definition of the Drinfeld center of a general monoidal 2-category, we refer the reader to \cite{KTZ}.

An object of $\mathscr{Z}(\mathfrak{C})$ consists of an object $Z$ of $\mathfrak{C}$ equipped with a half braiding, that is an adjoint 2-natural equivalence $b_Z$ given on $C$ in $\mathfrak{C}$ by $$(b_Z)_C:Z\Box C\xrightarrow{\simeq} C\Box Z,$$ together with an invertible modification $R_Z$ given on $C, D$ in $\mathfrak{C}$ by 
$$\begin{tikzcd}
Z\Box C\Box D \arrow[rr, "b_Z"] \arrow[rd, "b_Z1"'] & {} \arrow[d,Rightarrow, "R_Z",shorten > = 0.5ex] & C\Box D\Box Z \\
                                             & C\Box Z\Box D \arrow[ru, "1b_Z"']   &    
\end{tikzcd}$$

\noindent such that, for every objects $C,D,E$ in $\mathfrak{C}$, we have

\newlength{\drinfeld}
\settoheight{\drinfeld}{\includegraphics[width=37.5mm]{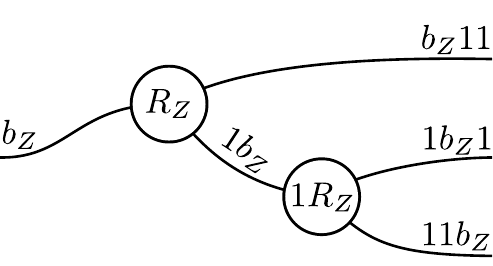}}

\begin{equation}\label{eqn:halfbraidedobject}
\begin{tabular}{@{}ccc@{}}

\includegraphics[width=37.5mm]{Pictures/center/halfbraidedobject1.pdf} & \raisebox{0.45\drinfeld}{$=$} &
\includegraphics[width=37.5mm]{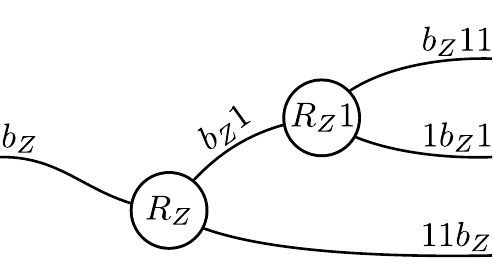}

\end{tabular}
\end{equation}

\noindent in $Hom_{\mathfrak{C}}(ZCDE, CDEZ)$.

A 1-morphism in $\mathscr{Z}(\mathfrak{C})$ from $(Y,b_Y,R_Y)$ to $(Z,b_Z,R_Z)$ consists of a 1-morphism $f:Y\rightarrow Z$ in $\mathfrak{C}$ together with an invertible modification $R_{f}$ given on $C$ in $\mathfrak{C}$ by $$\begin{tikzcd}
Y\Box C \arrow[r, "b_{Y}"] \arrow[d, "f1"']            & C\Box Y \arrow[d, "1f"] \\
Z\Box C \arrow[r, "{b_{Z}}"'] \arrow[ru,Rightarrow, "{R_{f}}",shorten > = 2ex,shorten <=2ex] & C\Box Z                
\end{tikzcd}$$

\noindent such that, for every objects $C,D$ in $\mathfrak{C}$, we have

\settoheight{\drinfeld}{\includegraphics[width=37.5mm]{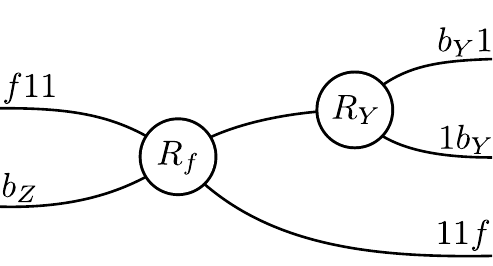}}

$$\begin{tabular}{@{}ccc@{}}

\includegraphics[width=37.5mm]{Pictures/center/halfbraided1morphism1.pdf} & \raisebox{0.45\drinfeld}{$=$} &
\includegraphics[width=45mm]{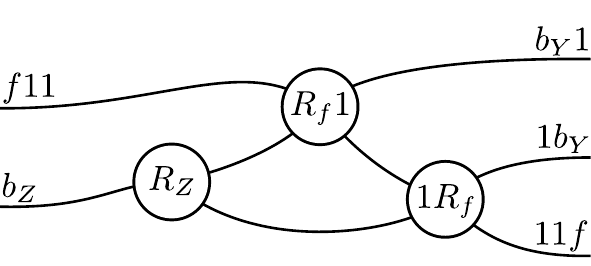}

\end{tabular}$$

\noindent in $Hom_{\mathfrak{C}}(YCD, CDZ)$.

A 2-morphism in $\mathscr{Z}(\mathfrak{C})$ from $(f,R_f)$ to $(g,R_g)$, two 1-morphisms from $(Y,b_y,R_Y)$ to $(Z,b_Z,R_Z)$, is a 2-morphism $\gamma:f\Rightarrow g$ in $\mathfrak{C}$ such that, for every object $C$ in $\mathfrak{C}$, we have

\settoheight{\drinfeld}{\includegraphics[width=30mm]{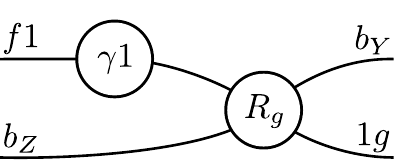}}

$$\begin{tabular}{@{}ccc@{}}

\includegraphics[width=30mm]{Pictures/center/halfbraided2morphism1.pdf} & \raisebox{0.45\drinfeld}{$=$} &
\includegraphics[width=30mm]{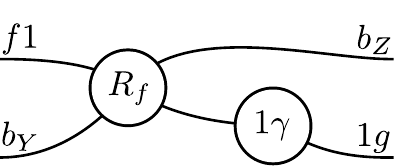}

\end{tabular}$$

\noindent in $Hom_{\mathfrak{C}}(YC, CZ)$.

The identity 1-morphism on the object $(Z,b_Z,R_Z)$ of $\mathscr{Z}(\mathfrak{C})$ is given by $(Id_Z, Id_{b_Z})$. Given two 1-morphisms $(f,R_f)$ from $(X,b_X,R_X)$ to $(Y,b_Y,R_Y)$ and $(g,R_g)$ from $(Y,b_Y,R_Y)$ to $(Z,b_Z,R_Z)$ in $\mathscr{Z}(\mathfrak{C})$ their composite is the 1-morphism $(g\circ f, (R_g\circ f1)\cdot (1g\circ R_f))$. This endows $\mathscr{Z}(\mathfrak{C})$ with the structure of a strict 2-category.

The monoidal structure of $\mathscr{Z}(\mathfrak{C})$ is given as follows. The monoidal unit is $I$, the monoidal unit of $\mathfrak{C}$, equipped with the identity adjoint 2-natural equivalence, and identity modification. The monoidal product of two objects $(Y,b_Y,R_Y)$ and $(Z,b_Z,R_Z)$ of $\mathscr{Z}(\mathfrak{C})$ is given by $$(Y,b_Y,R_Y)\Box (Z,b_Z,R_Z) = (Y\Box Z, (b_Y\Box 1) \circ (1\Box b_Z), R_{YZ}),$$ where $R_{YZ}$ is the invertible modification given on $C, D$ in $\mathfrak{C}$ by

\settoheight{\drinfeld}{\includegraphics[width=37.5mm]{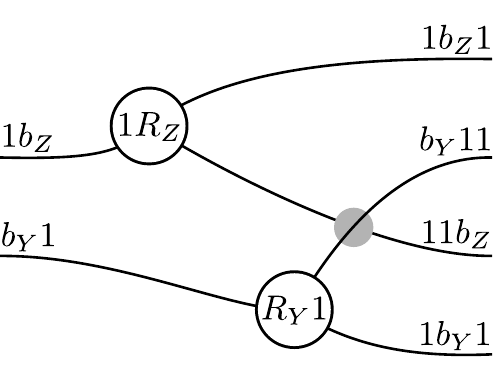}}

$$\begin{tabular}{@{}cc@{}}
 \raisebox{0.45\drinfeld}{$(R_{YZ})_{C,D}:=$} &
\includegraphics[width=37.5mm]{Pictures/center/RYZ.pdf}

\end{tabular}$$

\noindent in $Hom_{\mathfrak{C}}(YZCD,CDYZ)$. The monoidal product of two 1-morphisms $(f,R_f)$ and $(g,R_g)$ of $\mathscr{Z}(\mathfrak{C})$ is given by $(f,R_f)\Box (g,R_g) = (f\Box g,(R_f\Box g)\cdot (f\Box R_g)).$ Finally, the monoidal product of two 2-morphisms $\gamma$ and $\delta$ in $\mathscr{Z}(\mathfrak{C})$ is given by $\gamma\Box \delta$. Using the interchange $\phi^{\Box}$ of the monoidal product of $\mathfrak{C}$, we can upgrade the above assignment to a 2-functor that defines a strict cubical monoidal structure on $\mathscr{Z}(\mathfrak{C})$ as in \cite{Cr}. Furthermore, it follows from the definitions that the forgetful 2-functor $\mathscr{Z}(\mathfrak{C})\rightarrow \mathfrak{C}$ is monoidal.

Finally, in the notations of \cite{SP}, the braiding on $\mathscr{Z}(\mathfrak{C})$ is given as follows. For any two objects $(Y,b_Y,R_Y)$, and $(Z,b_Z,R_Z)$ of $\mathscr{Z}(\mathfrak{C})$, we define the braiding $b$ by $b_{Y,Z}:= (b_Y)_Z$. Further, $b$ is promoted to an adjoint 2-natural equivalence in an obvious way using the definitions of the objects and 1-morphisms in $\mathscr{Z}(\mathfrak{C})$. For any three objects $(X,b_X,R_X)$, $(Y,b_Y,R_Y)$, and $(Z,b_Z,R_Z)$ of $\mathscr{Z}(\mathfrak{C})$, we also define invertible modifications $R_{X,Y,Z}:= (R_X)_{Y,Z}$ and $S_{X,Y,Z}:= Id_{XYZ}$. It follows from \cite{Cr} that these assignments define a braiding on $\mathscr{Z}(\mathfrak{C})$.

\subsection{Elementary Properties}

As $\mathfrak{C}$ is a monoidal 2-category, it has a canonical structure of $\mathfrak{C}$-$\mathfrak{C}$-bimodule 2-category (see definition 2.1.1 of \cite{D4}). We begin by the following familiar observation, which will be generalized in proposition \ref{prop:DrinfeldIdentification} below.

\begin{Lemma}\label{lem:module2funcenter}
There is an equivalence $\mathscr{Z}(\mathfrak{C})\simeq \mathbf{End}_{\mathfrak{C}\mathrm{-}\mathfrak{C}}(\mathfrak{C})$ of monoidal 2-categories.
\end{Lemma}
\begin{proof}
This is a straightforward verification. We leave the details to the keen reader.
\end{proof}

Using $\mathfrak{C}^{mop}$ to denote $\mathfrak{C}$ with the opposite monoidal structure, we may equivalently think of $\mathfrak{C}$ as a $\mathfrak{C}$-$\mathfrak{C}$-bimodule 2-category or as a left $\mathfrak{C}\times\mathfrak{C}^{mop}$-module 2-category. By combining the above lemma with propositions 4.2.3 and 4.2.4 of \cite{D8}, we therefore obtain the following result.

\begin{Corollary}\label{cor:Drinfeldrigid}
If $\mathfrak{C}$ is a rigid monoidal 2-category, then its Drinfeld center $\mathscr{Z}(\mathfrak{C})$ is a rigid monoidal 2-category.
\end{Corollary}

We also record the following lemma, where $\mathscr{Z}(\mathfrak{C})^{rev}$ denotes the braided monoidal 2-category obtained from $\mathscr{Z}(\mathfrak{C})$ by using the opposite braiding. More precisely, given $(Y,b_Y,R_Y)$ and $(Z,b_Z,R_Z)$ in $\mathscr{Z}(\mathfrak{C})^{rev}$, the braiding on the objects is given by $b^{\bullet}_{Y,Z}:Y\Box Z\rightarrow Z\Box Y$, the prescribed adjoint equivalence of $b_{Y,Z}$.

\begin{Lemma}\label{lem:CenterOpRev}
There is an equivalence $\mathscr{Z}(\mathfrak{C}^{mop})\simeq \mathscr{Z}(\mathfrak{C})^{rev}$ of braided monoidal 2-categories.
\end{Lemma}
\begin{proof}
This follows by direct inspection.
\end{proof}

From now on, we will work over a fixed field $\mathds{k}$. In particular, we will assume that $\mathfrak{C}$ is a monoidal $\mathds{k}$-linear 2-category.

\begin{Lemma}\label{lem:CenterCauchyComplete}
Let $\mathfrak{C}$ be a Cauchy complete monoidal 2-category. Then, $\mathscr{Z}(\mathfrak{C})$ is Cauchy complete.
\end{Lemma}
\begin{proof}
Let $(Y,e,\mu,\delta)$ be a 2-condensation monad in $\mathscr{Z}(\mathfrak{C})$. By hypothesis, the underlying 2-condensation monad in $\mathfrak{C}$ can be split by a 2-condensation $(Y,Z,f,g,\phi,\gamma)$ and 2-isomorphism $\theta:g\circ f\cong e$. We begin by upgrading $Z$ to an object of $\mathscr{Z}(\mathfrak{C})$. To this end, recall that splittings of 2-condensation monads are preserved by all 2-functors. In particular, for any $C$ in $\mathfrak{C}$, $b_{Y,C}:Y\Box C\simeq C\Box Y$ induces an equivalence between two 2-condensation monads. By the 2-universal property of the splitting of 2-condensation monads (see theorem 2.3.2 of \cite{GJF}), this induces an equivalence $b_{Z,C}:Z\Box C\simeq C\Box Z$. Further, thanks to the 2-universal property these 1-morphisms assemble to given an adjoint 2-natural equivalence $b_Z$. Likewise, the invertible modification $R_Y$ induces an invertible modification $R_Z$ on $Z$, which satisfies equation (\ref{eqn:halfbraidedobject}). Similarly, one can endow both $f$ and $g$ with the structures of 1-morphisms in $\mathscr{Z}(\mathfrak{C})$, and check that $\phi$, $\gamma$, and $\theta$ are 2-morphisms in $\mathscr{Z}(\mathfrak{C})$. This concludes the proof.
\end{proof}

Let $\mathfrak{C}$ and $\mathfrak{D}$ be two Cauchy complete monoidal linear 2-categories. It follows from theorem 5.2 of \cite{D3} that the completed tensor product $\mathfrak{C}\widehat{\otimes}\mathfrak{D}$ inherits a monoidal structure. It is therefore natural to attempt to compare its Drinfeld center with that of $\mathfrak{C}$ and $\mathfrak{D}$. In this direction, we prove the following technical result.

\begin{Lemma}\label{lem:CenterProductWeak}
Let $\mathfrak{C}$ and $\mathfrak{D}$ be two Cauchy complete monoidal 2-categories. Then, there is a canonical braided monoidal 2-functor $\mathscr{Z}(\mathfrak{C})\widehat{\otimes} \mathscr{Z}(\mathfrak{C})\rightarrow \mathscr{Z}(\mathfrak{C}\widehat{\otimes} \mathfrak{D})$ that commutes with the forgetful monoidal 2-functors to $\mathfrak{C}\widehat{\otimes} \mathfrak{D}$ up to equivalence.
\end{Lemma}
\begin{proof}
We claim that there is a bilinear 2-functor $$\begin{tabular}{ccc}$\mathscr{Z}(\mathfrak{C})\times\mathscr{Z}(\mathfrak{D})$& $\rightarrow$ & $\mathscr{Z}(\mathfrak{C}\widehat{\otimes}\mathfrak{D}).$\\ $(Y,Z)$&$\mapsto$ & $Y\widehat{\otimes} Z$\end{tabular}$$ Namely, given $Y$ in $\mathscr{Z}(\mathfrak{C})$ and $Z$ in $\mathscr{Z}(\mathfrak{D})$, we can equip the object $Y\widehat{\otimes} Z$ of $\mathfrak{C}\widehat{\otimes}\mathfrak{D}$ with the adjoint 2-natural equivalence $b_Y\widehat{\otimes} b_Z$, which is defined for all objects of $\mathfrak{C}\widehat{\otimes}\mathfrak{D}$ of the form $C\widehat{\otimes} D$. But, by the construction given in the proof of proposition 3.4 of \cite{D3}, $\mathfrak{C}\widehat{\otimes} \mathfrak{D}$ is the Cauchy completion of its full sub-2-category on the objects of the form $C\widehat{\otimes} D$ with $C$ in $\mathfrak{C}$ and $D$ in $\mathfrak{D}$. Thus, thanks to the 3-universal property of the Cauchy completion (see \cite{D1}), $b_Y\widehat{\otimes} b_Z$ can be canonically extended to an adjoint 2-natural equivalence defined on all the objects of $\mathfrak{C}\widehat{\otimes}\mathfrak{D}$. An analogous argument shows that the invertible modification $R_Y\widehat{\otimes} R_Z$ can be canonically extended to all of $\mathfrak{C}\widehat{\otimes}\mathfrak{D}$. Further, it follows from its construction that this extension satisfies equation (\ref{eqn:halfbraidedobject}), so that $Y\widehat{\otimes} Z$ is indeed an object of $\mathscr{Z}(\mathfrak{C}\widehat{\otimes}\mathfrak{D})$. This assignment is extended to morphisms using the same techniques.
 
Thanks to lemma \ref{lem:CenterCauchyComplete} above and the 3-universal property of the completed tensor product obtained in proposition 3.4 of \cite{D3}, the bilinear 2-functor above induces a linear 2-functor $\mathscr{Z}(\mathfrak{C})\widehat{\otimes}\mathscr{Z}(\mathfrak{D})\rightarrow \mathscr{Z}(\mathfrak{C}\widehat{\otimes}\mathfrak{D})$. By construction, this is a braided monoidal 2-functor, and it is compatible with the forgetful monoidal 2-functors.
\end{proof}

\subsection{Morita Invariance}

Let us now fix $\mathfrak{C}$ a multifusion 2-category over an algebraically closed field $\mathds{k}$ of characteristic zero, and $\mathfrak{M}$ a separable left $\mathfrak{C}$-module 2-category. Thanks to proposition 2.2.8 of \cite{D4}, we may assume without loss of generality that $\mathfrak{C}$ and $\mathfrak{M}$ are strict cubical in the sense of definition 2.2.7 of \cite{D4}. Now, it is shown in proposition 4.1.10 of \cite{D8} that $\mathfrak{M}$ can be canonically viewed as a left $\mathfrak{C}\times\mathfrak{C}_{\mathfrak{M}}^*$-module 2-category. Further, this left action is multilinear, so that we may view $\mathfrak{M}$ as a left $\mathfrak{C}\boxtimes\mathfrak{C}_{\mathfrak{M}}^*$-module 2-category.

\begin{Proposition}\label{prop:DrinfeldIdentification}
Let $\mathfrak{M}$ be a faithful separable left $\mathfrak{C}$-module 2-category. Then, there is an equivalence $$(\mathfrak{C}\boxtimes\mathfrak{C}_{\mathfrak{M}}^*)_{\mathfrak{M}}^*\simeq \mathscr{Z}(\mathfrak{C})$$ of monoidal 2-categories.
\end{Proposition}
\begin{proof}
By definition, we have that $(\mathfrak{C}\boxtimes\mathfrak{C}_{\mathfrak{M}}^*)_{\mathfrak{M}}^*=\mathbf{End}_{\mathfrak{C}\boxtimes\mathfrak{C}_{\mathfrak{M}}^*}(\mathfrak{M})$. In fact, as the left action of $\mathfrak{C}\times\mathfrak{C}_{\mathfrak{M}}^*$ on $\mathfrak{M}$ is multilinear, it follows from the 3-universal property of the 2-Deligne tensor product obtained in \cite{D4} that $(\mathfrak{C}\boxtimes\mathfrak{C}_{\mathfrak{M}}^*)_{\mathfrak{M}}^*\simeq\mathbf{End}_{\mathfrak{C}\times\mathfrak{C}_{\mathfrak{M}}^*}(\mathfrak{M})$. Further, note that left $\mathfrak{C}\times\mathfrak{C}_{\mathfrak{M}}^*$-module 2-endofunctors on $\mathfrak{M}$ are exactly left $\mathfrak{C}_{\mathfrak{M}}^*$-module 2-endofunctors on $\mathfrak{M}$ equipped with a left $\mathfrak{C}$-module structure, whose coherence data consists of left $\mathfrak{C}_{\mathfrak{M}}^*$-module 2-natural transformations and left $\mathfrak{C}_{\mathfrak{M}}^*$-module modifications. More succinctly, $\mathfrak{C}\times\mathfrak{C}_{\mathfrak{M}}^*$-module 2-endofunctors are exactly 2-endofunctors equipped with commuting left actions by $\mathfrak{C}$ and $\mathfrak{C}_{\mathfrak{M}}^*$. Similar descriptions hold for left $\mathfrak{C}\times\mathfrak{C}_{\mathfrak{M}}^*$-module 2-natural transformations, as well as left $\mathfrak{C}\times\mathfrak{C}_{\mathfrak{M}}^*$-module modifications.

Given $(Z,b_Z,R_Z)$ in $\mathscr{Z}(\mathfrak{C})$, we can consider the 2-endofunctor $\mathbf{J}(Z)$ on $\mathfrak{M}$ given by $M\mapsto Z\Box M$. For any $F$ in $\mathfrak{C}_{\mathfrak{M}}^*$, the left $\mathfrak{C}$-module structure on $F$ yields an adjoint 2-natural equivalence $(\chi^F_{Z,M})^{\bullet}:F(Z\Box M)\simeq Z\Box F(M)$ and invertible modification, which endow $\mathbf{J}(Z)$ with a canonical left $\mathfrak{C}_{\mathfrak{M}}^*$-module structure. Furthermore, $\mathbf{J}(Z)$ carries a compatible left $\mathfrak{C}$-module structure given by $b^{\bullet}_{Z,C}\Box M: C\Box Z\Box M\simeq Z\Box C\Box M$. This assignment can be extended to morphisms in the obvious way, so that we have a 2-functor $\mathbf{J}:\mathscr{Z}(\mathfrak{C})\rightarrow \mathbf{End}_{\mathfrak{C}\times\mathfrak{C}_{\mathfrak{M}}^*}(\mathfrak{M})$. Further, it follows by inspection that $\mathbf{J}$ is monoidal.

Now, corollary 5.4.4 of \cite{D8} implies that the canonical monoidal 2-functor $\mathfrak{C}\rightarrow \mathbf{End}_{\mathfrak{C}_{\mathfrak{M}}^*}(\mathfrak{M})$ given by $C\mapsto \{M\mapsto C\Box M\}$ is an equivalence. In particular, objects of $\mathbf{End}_{\mathfrak{C}\times\mathfrak{C}_{\mathfrak{M}}^*}(\mathfrak{M})$ are identified with 2-endofunctors $M\mapsto Z\Box M$ for some $Z$ in $\mathfrak{C}$ equipped with a left action by $\mathfrak{C}$. More precisely, for every $C$ in $\mathfrak{C}$, we have a left $\mathfrak{C}_{\mathfrak{M}}^*$-module adjoint 2-natural equivalence $C\Box Z\Box M \simeq Z\Box C\Box M$. By the above equivalence of monoidal 2-categories, this corresponds exactly to an adjoint 2-natural equivalence $b_{Z,C}^{\bullet}:C\Box Z \simeq Z\Box C$. An analogous argument shows that the invertible modification witnessing the coherence of the left $\mathfrak{C}$-action corresponds precisely to the data of an invertible modification $R_Z$ satisfying (\ref{eqn:halfbraidedobject}). This shows that $\mathbf{J}$ is essentially surjective on objects. One proceeds similarly to show that $\mathbf{J}$ is essentially surjective on 1-morphisms and fully faithful on 2-morphisms. This concludes the proof.
\end{proof}

We are now ready to prove the main theorem of this section.

\begin{Theorem}\label{thm:MoritaInvarianceDrinfeldCenter}
Let $\mathfrak{C}$ and $\mathfrak{D}$ be Morita equivalent multifusion 2-categories, then there is an equivalence $\mathscr{Z}(\mathfrak{C})\simeq\mathscr{Z}(\mathfrak{D})$ of braided monoidal 2-categories.
\end{Theorem}
\begin{proof}
Thanks to theorem 5.4.3 of \cite{D8}, there exists a faithful separable left $\mathfrak{C}$-module 2-category $\mathfrak{M}$ and an equivalence $\mathfrak{D}\simeq (\mathfrak{C}_{\mathfrak{M}}^*)^{mop}$ of monoidal 2-categories. Thus, by lemma \ref{lem:CenterOpRev} it is enough to prove that $\mathscr{Z}(\mathfrak{C}^*_{\mathfrak{M}})^{rev}\simeq \mathscr{Z}(\mathfrak{C})$. Thanks to the proof of proposition \ref{prop:DrinfeldIdentification}, we know that the canonical monoidal 2-functor $\mathbf{J}:\mathscr{Z}(\mathfrak{C})\rightarrow \mathbf{End}_{\mathfrak{C}\times\mathfrak{C}_{\mathfrak{M}}^*}(\mathfrak{M})$ that sends $Z$ in $\mathscr{Z}(\mathfrak{C})$ to $M\mapsto Z\Box M$ is an equivalence. But, thanks to corollary corollary 5.4.4 of \cite{D8}, we have that $ \mathfrak{D}_{\mathfrak{M}}^*\simeq \mathfrak{C}^{mop}$ as monoidal 2-categories. This implies that the canonical monoidal 2-functor $\mathbf{K}:\mathscr{Z}(\mathfrak{C}_{\mathfrak{M}}^*)\rightarrow \mathbf{End}_{\mathfrak{C}\times\mathfrak{C}_{\mathfrak{M}}^*}(\mathfrak{M})$ that sends $F$ in $\mathscr{Z}(\mathfrak{C}_{\mathfrak{M}}^*)$ to $M\mapsto F(M)$ is an equivalence.

Now, given any object $Z$ in $\mathscr{Z}(\mathfrak{C})$, the left $\mathfrak{C}$-module 2-endofunctor $M\mapsto Z\Box M$ can be canonically upgraded to an object of $\mathscr{Z}(\mathfrak{C}_{\mathfrak{M}}^*)$. Namely, for every object $F$ of $\mathfrak{C}_{\mathfrak{M}}^*$, the left $\mathfrak{C}$-module structure on $F$ provides us with an adjoint 2-natural equivalence $\chi^F_{Z,M}:Z\Box F(M)\simeq F(Z\Box M)$, and an invertible modification satisfying (\ref{eqn:halfbraidedobject}). This define a half braiding on $M\mapsto Z\Box M$, and yields a monoidal 2-functor $\mathbf{L}:\mathscr{Z}(\mathfrak{C})\rightarrow \mathscr{Z}(\mathfrak{C}_{\mathfrak{M}}^*)$. Further, observe that the following diagram of monoidal 2-functors is weakly commutative: $$\begin{tikzcd}
\mathscr{Z}(\mathfrak{C}) \arrow[rd, "\mathbf{J}"'] \arrow[rr, "\mathbf{L}"] &                                                                       & \mathscr{Z}(\mathfrak{C}^*_{\mathfrak{M}}) \arrow[ld, "\mathbf{K}"] \\
                                                & \mathbf{End}_{\mathfrak{C}\times\mathfrak{C}_{\mathfrak{M}}^*}(\mathfrak{M}). &         
\end{tikzcd}$$ It follows from the first part of the proof that both $\mathbf{J}$ and $\mathbf{K}$ are equivalences, so that $\mathbf{L}$ is an equivalence of monoidal 2-categories. 

Finally, let $(Y,b_Y,R_Y)$ and $(Z,b_Z,R_Z)$ be objects of $\mathscr{Z}(\mathfrak{C})$. By the definition recalled in section \ref{sub:definitioncenter} above, the braiding on $\mathbf{L}(Y)$ and $\mathbf{L}(Z)$ in $\mathscr{Z}(\mathfrak{C}_{\mathfrak{M}}^*)$ is given by $$b_{\mathbf{L}(Y),\mathbf{L}(Z)}(M) = \chi^{\mathbf{L}(Z)}_{Y,M}= b^{\bullet}_{Z,Y}\Box M: Y\Box Z\Box M\rightarrow Z\Box Y\Box M.$$ Thus, $\mathbf{L}$ yields an equivalence $\mathscr{Z}(\mathfrak{C})\rightarrow \mathscr{Z}(\mathfrak{C}_{\mathfrak{M}}^*)^{rev}$ of braided monoidal 2-categories. This concludes the proof.
\end{proof}

\begin{Remark}\label{rem:BrauerPicardCenter}
The proof of theorem \ref{thm:MoritaInvarianceDrinfeldCenter} constructs for every Morita equivalence between $\mathfrak{C}$ and $\mathfrak{D}$ an equivalence $\mathscr{Z}(\mathfrak{C})\simeq\mathscr{Z}(\mathfrak{D})$ of braided monoidal 2-categories. With $\mathfrak{C} = \mathfrak{D}$, we expect that this assignment defines a group homomorphism $\Phi:\mathrm{BrPic}(\mathfrak{C})\rightarrow Aut^{br}(\mathscr{Z}(\mathfrak{C}))$, from the group of Morita autoequivalences of $\mathfrak{C}$ to the group of braided monoidal autoequivalences of $\mathscr{Z}(\mathfrak{C})$. This is a partial categorification of theorem 1.1 of \cite{ENO2}. On the other hand, it is well-known that two fusion 1-categories are Morita equivalent if and only if their Drinfeld centers are braided equivalent (see theorem 3.1 of \cite{ENO3}). This is not the case for fusion 2-categories! Namely, it follows from lemma 2.16 and theorem 2.52 of \cite{JFR} that, for any non-degenerate braided fusion 1-category $\mathcal{B}$, there is an equivalence $\mathscr{Z}(\mathbf{Mod}(\mathcal{B}))\simeq \mathbf{2Vect}$ of braided fusion 2-category (see also proposition 4.17 of \cite{DN}). Further, by example 5.4.6 of \cite{D8} or theorem \ref{thm:WittEquivalence} below, $\mathbf{Mod}(\mathcal{B})$ is Morita equivalent to $\mathbf{2Vect}$ if and only if $\mathcal{B}$ is Witt equivalent to $\mathbf{Vect}$. But, the Witt group of non-degenerate braided fusion 1-categories is known to be non-trivial (see \cite{DMNO}).
\end{Remark}

\begin{Corollary}\label{cor:Drinfeldgraded2vectorspaces}
There is an equivalence of braided fusion 2-categories $$\mathscr{Z}(\mathbf{2Vect}_G)\simeq\mathscr{Z}(\mathbf{2Rep}(G)).$$
\end{Corollary}

\begin{Remark}
Corollary \ref{cor:Drinfeldgraded2vectorspaces} has one particularly noteworthy consequence, which we now explain. Namely, as $\mathscr{Z}(\mathbf{2Vect}_G)$ and $\mathscr{Z}(\mathbf{2Rep}(G))$ are equivalent as braided monoidal 2-categories, the associated (monoidal) 2-categories of algebras, and their homomorphisms are equivalent. In particular, this induces an equivalence between the full sub-2-categories on the rigid algebras. On one hand, rigid algebras in $\mathscr{Z}(\mathbf{2Vect}_G)$ are precisely $G$-crossed multifusion 1-categories. On the other hand, rigid algebras in $\mathscr{Z}(\mathbf{2Rep}(G))$ are exactly multifusion 1-categories $\mathcal{C}$ equipped with two braided monoidal functors $\mathbf{Rep}(G)\rightrightarrows \mathcal{Z}(\mathcal{C})$ lifting $\mathbf{Rep}(G)\rightarrow \mathcal{C}$. This correspondence appears to be new.

In addition, we also get an equivalence between the 2-categories of braided rigid algebras. More precisely, we get an equivalence between the 2-category of $G$-crossed braided multifusion 1-categories (see example \ref{ex:algebrasZ2VG}) and that of braided multifusion 1-categories equipped with a braided monoidal functor from $\mathbf{Rep}(G)$ (see table 1 of \cite{DN}). This is a well-known result in the theory of fusion 1-categories, which dates back to \cite{Ki} and \cite{Mu1}.
\end{Remark}

%% file: Connected.tex
\section{The Morita Theory of Connected Fusion 2-Categories}\label{sec:connected}

In this section, we study the notion of Morita equivalence between connected fusion 2-categories. Throughout, we work over an algebraically closed field $\mathds{k}$ of characteristic zero.

\subsection{Comparison with Witt equivalence}\label{sub:MoritaconnectedF2Cs}

We begin by introducing a notion of Witt equivalence between arbitrary braided fusion 1-categories.

Given a braided fusion 1-category $\mathcal{B}$ with braiding $\beta$. We use $\mathcal{Z}_{(2)}(\mathcal{B})$ to denote the symmetric center of $\mathcal{B}$, that is the full fusion sub-1-category of $\mathcal{B}$ on those objects $B$ for which $$\beta_{C,B}\circ \beta_{B,C} = Id_{B\otimes C}$$ for every object $C$ of $\mathcal{B}$. It follows from the definition that $\mathcal{Z}_{(2)}(\mathcal{B})$ is a symmetric fusion 1-category.

Let $\mathcal{C}$ be a $\mathcal{B}$-central fusion 1-category $\mathcal{C}$. We write $\beta$ for the braiding on $\mathcal{Z}(\mathcal{C})$, and $F:\mathcal{B}\rightarrow \mathcal{Z}(\mathcal{C})$ for the braided monoidal functor supplying $\mathcal{C}$ with its central structure. Then, we use $\mathcal{Z}(\mathcal{C},\mathcal{B})$ to denote the centralizer of the image of $\mathcal{B}$ in $\mathcal{Z}(\mathcal{C})$, that is the full fusion sub-1-category of $\mathcal{Z}(\mathcal{C})$ on those objects $Z$ for which $$\beta_{F(B),Z}\circ \beta_{Z,F(B)} = Id_{Z\otimes F(B)}$$ for every object $B$ of $\mathcal{B}$.

\begin{Definition}\label{def:Wittequivalence}
Let $\mathcal{B}_1$ and $\mathcal{B}_2$ be two braided fusion 1-categories whose symmetric centers are given by the symmetric fusion 1-category $\mathcal{E}$. We say $\mathcal{B}_1$ and $\mathcal{B}_2$ are Witt equivalent provided that there exists a fusion 1-category $\mathcal{C}$, together with a fully faithful braided embedding $\mathcal{B}_1\hookrightarrow \mathcal{Z}(\mathcal{C})$ and a braided monoidal equivalence $\mathcal{B}_2\simeq \mathcal{Z}(\mathcal{C},\mathcal{B}_1)^{rev}$.
\end{Definition}

Let $\mathcal{E}$ be a symmetric fusion 1-category. A braided fusion 1-category over $\mathcal{E}$, also known as an $\mathcal{E}$-central braided fusion 1-category, is a pair $(\mathcal{B},F)$ consisting of a braided fusion 1-category $\mathcal{B}$ and a braided embedding $F:\mathcal{E}\hookrightarrow \mathcal{B}$. We say that $(\mathcal{B},F)$, a braided fusion 1-category over $\mathcal{E}$, is non-degenerate if $F$ induces an equivalence $\mathcal{E}\simeq\mathcal{Z}_{(2)}(\mathcal{E})$. In \cite{DNO}, the authors introduced a notion of Witt equivalence over $\mathcal{E}$ between non-degenerate braided fusion 1-categories over $\mathcal{E}$. More precisely, $(\mathcal{B}_1,F_1)$ and $(\mathcal{B}_2,F_2)$ are Witt equivalent over $\mathcal{E}$ if there exists a fusion 1-category $\mathcal{C}$, together with a fully faithful braided embedding $\mathcal{B}_1\hookrightarrow \mathcal{Z}(\mathcal{C})$ and a braided monoidal equivalence $\mathcal{B}_2\simeq \mathcal{Z}(\mathcal{C},\mathcal{B}_1)^{rev}$ such that the two braided functors $\mathcal{E}\rightarrow\mathcal{Z}(\mathcal{C})$ are equivalent. We now explain how this is related to the notion of Witt equivalence we have introduced above.

\begin{Lemma}\label{lem:comparisonWittequivalenceoverE}
Let $\mathcal{B}_1$ and $\mathcal{B}_2$ be two braided fusion 1-categories whose symmetric centers are given by the symmetric fusion 1-category $\mathcal{E}$. Then, $\mathcal{B}_1$ and $\mathcal{B}_2$ are Witt equivalent in the sense of definition \ref{def:Wittequivalence} if and only if there exists identifications $F_1:\mathcal{E}\simeq \mathcal{Z}_{(2)}(\mathcal{B}_1)$ and $F_2:\mathcal{E}\simeq \mathcal{Z}_{(2)}(\mathcal{B}_2)$ such that $(\mathcal{B}_1,F_1)$ and $(\mathcal{B}_2,F_2)$ are Witt equivalent over $\mathcal{E}$ as in definition 5.1 of \cite{DNO}.
\end{Lemma}
\begin{proof}
Let us fix an identification $F_1:\mathcal{E}\simeq \mathcal{Z}_{(2)}(\mathcal{B}_1)$. If there exists a fusion 1-category $\mathcal{C}$, together with a fully faithful braided embedding $\mathcal{B}_1\hookrightarrow \mathcal{Z}(\mathcal{C})$ and a braided monoidal equivalence $\mathcal{B}_2\simeq \mathcal{Z}(\mathcal{C},\mathcal{B}_1)^{rev}$, then there is a canonical identification $F_2:\mathcal{E}\simeq \mathcal{Z}_{(2)}(\mathcal{B}_2)$ such that $$\mathcal{B}_1\boxtimes_{\mathcal{E}}\mathcal{B}_2^{rev}\simeq \mathcal{B}_1\boxtimes_{\mathcal{E}}\mathcal{Z}(\mathcal{C},\mathcal{B}_1)\simeq \mathcal{Z}(\mathcal{C},\mathcal{E})$$ by proposition 4.3 of \cite{DNO}. This shows that $(\mathcal{B}_1,F_1)$ and $(\mathcal{B}_2^{rev},F_2^{rev})$ define opposite elements in the Witt group $\mathcal{W}(\mathcal{E})$ in the sense of \cite{DNO}, so that $(\mathcal{B}_1,F_1)$ and $(\mathcal{B}_2,F_2)$ are Witt equivalent over $\mathcal{E}$ in the sense of definition 5.1 of \cite{DNO}. The converse follows by running the argument in reverse.
\end{proof}

\begin{Remark}
We emphasize that Witt equivalence over $\mathcal{E}$ does not in general coincide with the notion of Witt equivalence we have introduced above. Namely, given a braided fusion 1-category $\mathcal{B}$ and identifications $F_1,F_2:\mathcal{E}\simeq \mathcal{Z}_{(2)}(\mathcal{B})$, then $(\mathcal{B},F_1)$ and $(\mathcal{B},F_2)$ are Witt equivalent over $\mathcal{E}$ if and only if there exists $E$, an autoequivalence of $\mathcal{B}$ as a braided fusion 1-category, such that $F_1$ and $E\circ F_2$ are equivalent as braided functors. But,
there exists braided fusion 1-categories $\mathcal{B}$ such that the canonical group homomorphism $Aut^{br}(\mathcal{B})\rightarrow Aut^{br}(\mathcal{Z}_{(2)}(\mathcal{B}))$ is not surjective \cite{N:private}. For instance, given a prime $p>3$, the braided fusion 1-category $\mathbf{Vect}_{\mathbb{Z}/p^2\mathbb{Z}}^q$ associated to the quadratic form $x\mapsto (-1)^{x^2}$ exhibits such behaviour.
\end{Remark}

We can now state the main theorem of this section, which completely characterizes which connected fusion 2-categories are Morita equivalent. In particular, it generalizes example 5.4.6 of \cite{D8}. The proof relies on the technical results derived in section \ref{sub:technical} below.

\begin{Theorem}\label{thm:WittEquivalence}
Let $\mathcal{B}_1$ and $\mathcal{B}_2$ be two braided fusion 1-categories. Then, $\mathbf{Mod}(\mathcal{B}_1)$ and $\mathbf{Mod}(\mathcal{B}_2)$ are Morita equivalent fusion 2-categories if and only if $\mathcal{B}_1$ and $\mathcal{B}_2$ have the same symmetric center $\mathcal{E}$ and they are Witt equivalent.
\end{Theorem}

\begin{proof}
Assume that $\mathbf{Mod}(\mathcal{B}_1)$ and $\mathbf{Mod}(\mathcal{B}_2)$ are Morita equivalent fusion 2-categories. Then, it follows from theorem \ref{thm:MoritaInvarianceDrinfeldCenter} that $$\mathscr{Z}(\mathbf{Mod}(\mathcal{B}_1))\simeq \mathscr{Z}(\mathbf{Mod}(\mathcal{B}_2))$$ as braided fusion 2-categories. Thanks to lemma 2.16 of \cite{JFR} or corollary \ref{cor:omegacenter} below, we find that $$\mathcal{Z}_{(2)}(\mathcal{B}_1)\simeq \Omega\mathscr{Z}(\mathbf{Mod}(\mathcal{B}_1))\simeq \Omega\mathscr{Z}(\mathbf{Mod}(\mathcal{B}_2))\simeq \mathcal{Z}_{(2)}(\mathcal{B}_2)$$ as symmetric fusion 1-categories. We denote by $\mathcal{E}$ this symmetric fusion 1-category. Now, as $\mathbf{Mod}(\mathcal{B}_1)$ and $\mathbf{Mod}(\mathcal{B}_2)$ are Morita equivalent fusion 2-categories, there exists a fusion 1-category $\mathcal{C}$ equipped with a braided monoidal functor $F:\mathcal{B}_1\rightarrow \mathcal{Z}(\mathcal{C})$, and a monoidal equivalence $\mathbf{Bimod}_{\mathbf{Mod}(\mathcal{B}_1)}(\mathcal{C})\simeq \mathbf{Mod}(\mathcal{B}_2)$ of fusion 2-categories. Let us write $\mathcal{A}$ for the image of $\mathcal{B}_1$ in $\mathcal{Z}(\mathcal{C})$, which is a braided fusion 1-category by definition. Then, it follows from proposition 4.3 of \cite{DNO} and from the definitions that $$\mathcal{Z}_{(2)}(\mathcal{A}) \simeq \mathcal{Z}_{(2)}(\mathcal{Z}(\mathcal{C},\mathcal{A}))\simeq \mathcal{Z}_{(2)}(\mathcal{Z}(\mathcal{C},\mathcal{B}_1))$$ as symmetric fusion 1-categories. On the other hand, by lemma \ref{lem:dualconnectedcomponent} below, we find that $\mathcal{Z}(\mathcal{C},\mathcal{B}_1)^{rev}\simeq \mathcal{B}_2$ as braided fusion 1-categories. This implies that $\mathcal{Z}_{(2)}(\mathcal{A})\simeq \mathcal{E}$. Thence, it follows from corollary 3.24 of \cite{DMNO} that $F$ is fully faithful, so that $\mathcal{B}_1$ and $\mathcal{B}_2$ are Witt equivalent.

On the other hand, if $\mathcal{B}_1$ and $\mathcal{B}_2$ are Witt equivalent, then, by definition, there exists a $\mathcal{B}_1$-central fusion 1-category $\mathcal{C}$ such that $\mathcal{B}_1\rightarrow \mathcal{Z}(\mathcal{C})$ is a fully faithful embedding, and an equivalence $\mathcal{Z}(\mathcal{C},\mathcal{B}_1)^{rev}\simeq \mathcal{B}_2$ of braided fusion 1-categories. By corollary \ref{cor:dualwittequivalent} below, the $\mathcal{B}_1$-central fusion 1-category $\mathcal{C}$ witnesses the desired Morita equivalence.
\end{proof}

\begin{Corollary}\label{cor:dualconnected}
Let $\mathcal{B}$ be a braided fusion 1-category, and $\mathcal{C}$ a fusion 1-category with $\mathcal{B}$-central structure given by $F:\mathcal{B}\rightarrow\mathcal{Z}(\mathcal{C})$. Then, $\mathbf{Bimod}_{\mathbf{Mod}(\mathcal{B})}(\mathcal{C})$ is connected if and only if $F$ is fully faithful.
\end{Corollary}

\begin{Remark}
More generally, it would be interesting to understand the connected components of the fusion 2-category $\mathbf{Bimod}_{\mathbf{Mod}(\mathcal{B})}(\mathcal{C})$. For instance, with $\mathcal{B}=\mathbf{Rep}(G)$ for some finite group $G$, and $\mathcal{C}=\mathbf{Vect}$, it was shown in example 5.1.9 of \cite{D8} that $\mathbf{Bimod}_{\mathbf{Mod}(\mathcal{B})}(\mathcal{C})\simeq \mathbf{2Vect}_G$, which has $|G|$ connected components. On the other hand, with $\mathcal{B}=\mathbf{Rep}(G\times G^{op})$ and $\mathcal{C}=\mathbf{Rep}(G)$, it follows from lemma \ref{lem:module2funcenter} and corollary \ref{cor:Drinfeldgraded2vectorspaces} that $\mathbf{Bimod}_{\mathbf{Mod}(\mathcal{B})}(\mathcal{C})\simeq \mathscr{Z}(\mathbf{2Vect}_G)$. It was shown in \cite{KTZ} that the connected components of the underlying finite semisimple 2-category of $\mathscr{Z}(\mathbf{2Vect}_G)$ are parameterised by the conjugacy classes of $G$.
\end{Remark}

The technical results used in the proof of the above theorem will be established below. But first, we record the following corollary. Let us fix $\mathcal{E}$ a symmetric fusion 1-category. A non-degenerate braided fusion 1-category $\mathcal{B}$ equipped with a fully faithful embedding of $\mathcal{E}$ is called minimal if the centralizer of $\mathcal{E}$ in $\mathcal{B}$ is exactly $\mathcal{E}$. The set of minimal non-degenerate extensions of $\mathcal{E}$ forms a group (see \cite{LKW}), which is denoted by $Mext(\mathcal{E})$. We say that a minimal non-degenerate extension $\mathcal{B}$ of $\mathcal{E}$ is Witt-trivial if the class of the non-degenerate braided fusion 1-category $\mathcal{B}$ is trivial in the Witt group $\mathcal{W}$. This is equivalent to requiring that $\mathcal{B}$ is braided equivalent to the Drinfeld center of a fusion 1-category. We note that the set of Witt-trivial minimal non-degenerate extensions of $\mathcal{E}$ is a subgroup of $Mext(\mathcal{E})$ thanks to theorem 3.6.20 of \cite{Sun}.

\begin{Corollary}
Let $\mathcal{E}$ be a symmetric fusion 1-category. Then, there is a bijective correspondence between Morita autoequivalences of $\mathbf{Mod}(\mathcal{E})$ and pairs consisting of a Witt-trivial minimal non-degenerate extension of $\mathcal{E}$ and an equivalence class of symmetric monoidal functor $\mathcal{E}\simeq\mathcal{E}$.
\end{Corollary}
\begin{proof}
It follows readily from corollary \ref{cor:dualconnected} that every Morita autoequivalence of $\mathbf{Mod}(\mathcal{E})$ is given by an $\mathcal{E}$-central fusion 1-category $\mathcal{C}$ such that $\mathcal{Z}(\mathcal{C},\mathcal{E})=\mathcal{E}$ together with a braided autoequivalence of $\mathcal{E}$. On the other hand, any two $\mathcal{E}$-central fusion 1-categories $\mathcal{C}_1$ and $\mathcal{C}_2$ with $\mathcal{Z}(\mathcal{C}_1,\mathcal{E})=\mathcal{E}=\mathcal{Z}(\mathcal{C}_2,\mathcal{E})$ can induce the same Morita equivalence only if $\mathbf{Mod}(\mathcal{C}_1)$ and $\mathbf{Mod}(\mathcal{C}_2)$ are equivalent as left $\mathbf{Mod}(\mathcal{E})$-module 2-categories. This implies that $\mathcal{C}_1$ and $\mathcal{C}_2$ are Morita equivalent fusion 1-categories so that we have an equivalence $\mathcal{Z}(\mathcal{C}_1)\simeq\mathcal{Z}(\mathcal{C}_2)$ of braided fusion 1-categories (see theorem 1.1 of \cite{ENO2}). Finally, $\mathbf{Mod}(\mathcal{C}_1)$ and $\mathbf{Mod}(\mathcal{C}_2)$ are equivalent as left $\mathbf{Mod}(\mathcal{E})$-module 2-categories if and only if the diagram of braided monoidal functors $$\begin{tikzcd}
                                      & \mathcal{E} \arrow[ld] \arrow[rd] &                            \\
\mathcal{Z}(\mathcal{C}_1) \arrow[rr, "\simeq"] &                                   & \mathcal{Z}(\mathcal{C}_2)
\end{tikzcd}$$ commutes up to braided natural equivalence. This is the case if and only if $\mathcal{Z}(\mathcal{C}_1)$ and $\mathcal{Z}(\mathcal{C}_2)$ are equivalent minimal non-degenerate extensions of $\mathcal{E}$. This concludes the proof.
\end{proof}

\begin{Remark}\label{rem:braidedequivalencescenter}
Let $G$ be a finite group. It follows from the above corollary and example 2.1 of \cite{N} that there is a bijection of sets $\mathrm{BrPic}(\mathbf{Mod}(\mathbf{Rep}(G)))\simeq H^3(G;\mathds{k}^{\times})\times Aut^{br}(\mathbf{Rep}(G)))$. Namely, every minimal non-degenerate extension of $\mathbf{Rep}(G)$ is Witt-trivial. In this case, we expect that the map $\Phi:\mathrm{BrPic}(\mathbf{Mod}(\mathbf{Rep}(G)))\rightarrow Aut^{br}(\mathscr{Z}(\mathbf{Mod}(\mathbf{Rep}(G))))$ of remark \ref{rem:BrauerPicardCenter} is an isomorphism of groups.

On the other hand, the above corollary shows that $\mathrm{BrPic}(\mathbf{2SVect})\simeq 0$. Namely, the group of minimal non-degenerate extensions of $\mathbf{SVect}$ is identified with $\mathbb{Z}/16\mathbb{Z}$ by example 2.2 of \cite{N}, and only one of the sixteen minimal non-degenerate extensions of $\mathbf{SVect}$ is Witt-trivial (see appendix A.3.2 of \cite{DGNO}). But, every minimal non-degenerate extension $\mathcal{C}$ of $\mathbf{SVect}$ induces a braided autoequivalence of $\mathscr{Z}(\mathbf{2SVect})$. In detail, the left $\mathbf{2SVect}$-module 2-category $\mathbf{Mod}(\mathcal{C})$ induces a braided monoidal 2-functor $\mathscr{Z}(\mathbf{2SVect})\simeq \mathscr{Z}(\mathbf{2SVect}\boxtimes \mathbf{Mod}(\mathcal{C}))$, and, thanks to corollary \ref{cor:center2Deligne} below, the right-hand side is canonically equivalent to $\mathscr{Z}(\mathbf{2SVect})\boxtimes \mathscr{Z}(\mathbf{Mod}(\mathcal{C}))\simeq\mathscr{Z}(\mathbf{2SVect})$ as a braided monoidal 2-category. In fact, it was sketched in section 2.3 of \cite{JF2} that this construction induces a bijection of sets $Aut^{br}(\mathscr{Z}(\mathbf{2SVect}))\cong \mathbb{Z}/16\mathbb{Z}$. We therefore expect that the map $\Phi:\mathrm{BrPic}(\mathbf{2SVect})\rightarrow Aut^{br}(\mathscr{Z}(\mathbf{2SVect}))$ is not an isomorphism, so that the full version of theorem 1.1 of \cite{ENO2} can not be na\"ively categorified.
\end{Remark}

Let $\mathcal{E}$ be a symmetric fusion 1-category. The group $\mathcal{W}(\mathcal{E})$ of classes of non-degenerate braided fusion 1-categories over $\mathcal{E}$ up to Witt equivalence over $\mathcal{E}$ was introduced in \cite{DNO}. The group $Aut^{br}(\mathcal{E})$ acts on $\mathcal{W}(\mathcal{E})$. Namely, given $E:\mathcal{E}\simeq\mathcal{E}$ an equivalence of symmetric fusion 1-categories, and $(\mathcal{B},F)$ a non-degenerate braided fusion 1-category over $\mathcal{E}$, we can consider $(\mathcal{B},F\circ E)$. Further, it is clear that this assignment is compatible with Witt equivalence over $\mathcal{E}$. We write $\widehat{\mathcal{W}}(\mathcal{E})$ for the quotient of $\mathcal{W}(\mathcal{E})$ under the action of $Aut^{br}(\mathcal{E})$. The next result follows by combing lemma \ref{lem:comparisonWittequivalenceoverE} with theorem \ref{thm:Moritaconnected}.

\begin{Corollary}\label{cor:Moritaclassesconnected}
Morita equivalence classes of connected fusion 2-categories correspond precisely to pairs consisting of a symmetric fusion 1-category $\mathcal{E}$ and a class in $\widehat{\mathcal{W}}(\mathcal{E})$.
\end{Corollary}

\begin{Remark}\label{rem:centralMoritaequivalence}
Let us fix $\mathcal{E}$ a symmetric fusion 1-category. A $\mathbf{Mod}(\mathcal{E})$-central fusion 2-category, or a fusion 2-category over $\mathbf{Mod}(\mathcal{E})$, is a fusion 2-category $\mathfrak{C}$ equipped with a braided monoidal 2-functor $\mathbf{F}:\mathbf{Mod}(\mathcal{E})\rightarrow \mathscr{Z}(\mathfrak{C})$. A Morita equivalence
between two $\mathbf{Mod}(\mathcal{E})$-central fusion 2-category $(\mathfrak{C}_1,\mathbf{F}_1)$ and $(\mathfrak{C}_2,\mathbf{F}_1)$ is a Morita equivalence between the underlying fusion 2-categories $\mathfrak{C}_1$ and $\mathfrak{C}_2$ such that the induced diagram of braided monoidal 2-functors $$\begin{tikzcd}
& \mathbf{Mod}(\mathcal{E}) \arrow[ld, "\mathbf{F}_1"'] \arrow[rd, "\mathbf{F}_2"] & \\
\mathscr{Z}(\mathfrak{C}_1) \arrow[rr, "\simeq"] & & \mathscr{Z}(\mathfrak{C}_2)
\end{tikzcd}$$ commutes up to braided monoidal 2-natural equivalence. Now, the data of a $\mathbf{Mod}(\mathcal{E})$-central structure for a connected fusion 2-category $\mathbf{Mod}(\mathcal{B})$ corresponds exactly to a symmetric monoidal functor $F:\mathcal{E}\rightarrow \mathcal{Z}_{(2)}(\mathcal{B})$. Namely, we have $\Omega\mathscr{Z}(\mathbf{Mod}(\mathcal{B}))\simeq \mathbf{Mod}(\mathcal{Z}_{(2)}(\mathcal{B}))$, and any monoidal 2-functor $\mathbf{F}:\mathbf{Mod}(\mathcal{E})\rightarrow \mathscr{Z}(\mathbf{Mod}(\mathcal{B}))$ is completely determined by the braided monoidal functor $\Omega\mathbf{F}$. In particular, if $\mathcal{B}_1$ and $\mathcal{B}_2$ are braided fusion 1-categories, and $F_1:\mathcal{E}\simeq\mathcal{Z}_{(2)}(\mathcal{B}_1)$ and $F_2:\mathcal{E}\simeq\mathcal{Z}_{(2)}(\mathcal{B}_2)$ are equivalences of symmetric fusion 1-categories, it follows from theorem \ref{thm:WittEquivalence} that the $\mathbf{Mod}(\mathcal{E})$-central fusion 2-category $(\mathbf{Mod}(\mathcal{B}_1),\mathbf{Mod}(F_1))$ and $(\mathbf{Mod}(\mathcal{B}_2),\mathbf{Mod}(F_1))$ are Morita equivalent if and only if $(\mathcal{B}_1, F_1)$ and $(\mathcal{B}_2, F_2)$ are Witt equivalent over $\mathcal{E}$ in the sense of \cite{DNO}.
\end{Remark}

\subsection{Technical Results}\label{sub:technical}

We establish the technical lemmas that were used in the proof of theorem \ref{thm:WittEquivalence}.

\begin{Lemma}\label{lem:dualconnectedcomponent}
Let $\mathcal{B}$ be a braided fusion 1-category, and $\mathcal{C}$ a $\mathcal{B}$-central fusion 1-category. There is an equivalence $$\mathbf{Bimod}_{\mathbf{Mod}(\mathcal{B})}(\mathcal{C})^0\simeq \mathbf{Mod}(\mathcal{Z}(\mathcal{C},\mathcal{B})^{rev})$$ of monoidal 2-categories.
\end{Lemma}
\begin{proof}
We have seen in example \ref{ex:algebrasModB} above that there is an equivalence of 2-categories $$\mathbf{Bimod}_{\mathbf{Mod}(\mathcal{B})}(\mathcal{C})\simeq\mathbf{Mod}(\mathcal{C}^{mop}\boxtimes_{\mathcal{B}}\mathcal{C}).$$ Further, as observed in remark \ref{rem:balancedbimdoule1categories}, the 2-category on the left hand-side is the 2-category of finite semisimple $\mathcal{B}$-balanced $\mathcal{C}$-$\mathcal{C}$-bimodule 1-categories as considered in section 3.4 of \cite{Lau}. Now, $\mathcal{C}$ with its canonical $\mathcal{B}$-balanced $\mathcal{C}$-$\mathcal{C}$-bimodule structure is the monoidal unit of $\mathbf{Bimod}_{\mathbf{Mod}(\mathcal{B})}(\mathcal{C})$. Let us write $End^{\mathcal{B}}_{\mathcal{C}\textrm{-}\mathcal{C}}(\mathcal{C})$ for the braided fusion 1-category of endomorphisms of $\mathcal{C}$ in $\mathbf{Bimod}_{\mathbf{Mod}(\mathcal{B})}(\mathcal{C})$. Now, observe that the forgetful 2-functor $$\mathbf{Bimod}_{\mathbf{Mod}(\mathcal{B})}(\mathcal{C})\rightarrow\mathbf{Bimod}(\mathcal{C})$$ is monoidal. On the monoidal unit, this 2-functor induces a braided monoidal functor $End^{\mathcal{B}}_{\mathcal{C}\textrm{-}\mathcal{C}}(\mathcal{C})\rightarrow End_{\mathcal{C}\textrm{-}\mathcal{C}}(\mathcal{C})$. But, by example 5.3.7 of \cite{D8}, we know that there is an equivalence $End_{\mathcal{C}\textrm{-}\mathcal{C}}(\mathcal{C})\simeq \mathcal{Z}(\mathcal{C})^{rev}$ of braided fusion 1-categories. Moreover, it follows from definition 3.24 of \cite{Lau} that $End^{\mathcal{B}}_{\mathcal{C}\textrm{-}\mathcal{C}}(\mathcal{C})$ is the full fusion sub-1-category of $End_{\mathcal{C}\textrm{-}\mathcal{C}}(\mathcal{C})\simeq \mathcal{Z}(\mathcal{C})^{rev}$ on those objects centralizing the image of $\mathcal{B}^{rev}$ in $\mathcal{Z}(\mathcal{C})^{rev}$. Thus, we find that $End^{\mathcal{B}}_{\mathcal{C}\textrm{-}\mathcal{C}}(\mathcal{C})\simeq \mathcal{Z}(\mathcal{C},\mathcal{B})^{rev}$ as braided fusion 1-categories. Finally, the proof is concluded by appealing to proposition 2.4.7 of \cite{D2}.
\end{proof}

\begin{Remark}\label{rem:Laughwitz}
As already mentioned in remark \ref{rem:balancedbimdoule1categories}, \cite{Lau} works with $\mathcal{B}$-augmented tensor 1-categories, which are $\mathcal{B}$-balanced tensor 1-categories equipped with additional structure. Nonetheless, the results of section 3.4 of \cite{Lau} hold for any $\mathcal{B}$-central tensor 1-category. On the other hand, this is not the case for all the results of section 3 of \cite{Lau}. More precisely, corollary 3.46 of \cite{Lau} does not hold for an arbitrary $\mathcal{B}$-central tensor 1-category. Namely, if $\mathcal{B}=\mathbf{Rep}(G)$, and $\mathcal{C}=\mathbf{Vect}$, then $\mathcal{C}^{mop}\boxtimes_{\mathcal{B}}\mathcal{C} \simeq \boxplus_{g\in G}\mathbf{Vect}$ and $\mathcal{C}$ is not a faithful $\mathcal{C}^{mop}\boxtimes_{\mathcal{B}}\mathcal{C}$-module 1-category.
\end{Remark}

We recover lemma 2.16 of \cite{JFR}.

\begin{Corollary}\label{cor:omegacenter}
Let $\mathcal{B}$ be a braided fusion 1-category, then we have $$\Omega\mathscr{Z}(\mathbf{Mod}(\mathcal{B}))\simeq \mathcal{Z}_{(2)}(\mathcal{B}).$$
\end{Corollary}
\begin{proof}
Thanks to lemma \ref{lem:module2funcenter} we can identify $\mathscr{Z}(\mathbf{Mod}(\mathcal{B}))$ with the dual to $\mathbf{Mod}(\mathcal{B}^{rev}\boxtimes\mathcal{B})$ with respect to $\mathbf{Mod}(\mathcal{B})$. It follows from corollary 4.4 of \cite{DNO} and theorem 3.10 (ii) of \cite{DGNO} that $\mathcal{Z}(\mathcal{B}, \mathcal{B}^{rev}\boxtimes\mathcal{B})\simeq \mathcal{Z}_{(2)}(\mathcal{B})$, which establishes the result.
\end{proof}

We now give a sufficient criterion for the fusion 2-category $\mathbf{Bimod}_{\mathbf{Mod}(\mathcal{B})}(\mathcal{C})$ to be connected.

\begin{Lemma}\label{lem:dualconnectedweak}
Let $\mathcal{B}$ be a braided fusion 1-category, and $\mathcal{C}$ a $\mathcal{B}$-central fusion 1-category such that the composite $\mathcal{B}\rightarrow\mathcal{C}$ is fully faithful. Then, $\mathbf{Bimod}_{\mathbf{Mod}(\mathcal{B})}(\mathcal{C})$ is a connected fusion 2-category.
\end{Lemma}
\begin{proof}
Let us identify $\mathbf{Bimod}_{\mathbf{Mod}(\mathcal{B})}(\mathcal{C})\simeq \mathbf{Mod}(\mathcal{C}^{mop}\boxtimes_{\mathcal{B}}\mathcal{C})$ as 2-categories. Thanks to proposition 2.3.5 of \cite{D1}, in order to prove that this finite semisimple 2-category is connected, it is enough to show that $\mathcal{C}^{mop}\boxtimes_{\mathcal{B}}\mathcal{C}$ is a fusion 1-category. But, the canonical functor $\mathcal{C}^{mop}\boxtimes\mathcal{C}\rightarrow \mathcal{C}^{mop}\boxtimes_{\mathcal{B}}\mathcal{C}$ is monoidal. Further, the image of $\mathcal{B}$ in $\mathcal{C}$ may be identified with $\mathcal{B}$ by hypothesis. Thus, the monoidal unit of $\mathcal{C}^{mop}\boxtimes_{\mathcal{B}}\mathcal{C}$ is given by the image of the monoidal unit $I\boxtimes I$ of $\mathcal{B}^{mop}\boxtimes\mathcal{B}$ under the canonical monoidal functor $\mathcal{B}^{mop}\boxtimes\mathcal{B}\rightarrow \mathcal{B}^{mop}\boxtimes_{\mathcal{B}}\mathcal{B} \simeq \mathcal{B}$. As $\mathcal{B}$ is a fusion 1-category by hypothesis, the proof is complete.
\end{proof}

We need the following slightly more general version of the previous lemma.

\begin{Lemma}\label{lem:dualconnectedstrong}
Let $\mathcal{B}$ be a braided fusion 1-category, and $\mathcal{C}$ a fusion 1-category with $\mathcal{B}$-central structure $F:\mathcal{B}\rightarrow\mathcal{Z}(\mathcal{C})$ given by a fully faithful braided monoidal functor. Then, $\mathbf{Bimod}_{\mathbf{Mod}(\mathcal{B})}(\mathcal{C})$ is a connected fusion 2-category.
\end{Lemma}
\begin{proof}
Let us consider the fusion 2-category $\mathbf{Bimod}_{\mathbf{Mod}(\mathcal{B})}(\mathcal{Z}(\mathcal{C}))$, where $\mathcal{Z}(\mathcal{C})$ is endowed with its canonical $\mathcal{B}$-central structure. Thanks to lemma \ref{lem:dualconnectedweak} above, it is connected. On the other hand, by theorem \ref{thm:dualfusion2category} above, there is an equivalence of fusion 2-categories $$\mathbf{Bimod}_{\mathbf{Mod}(\mathcal{B})}(\mathcal{Z}(\mathcal{C}))\simeq \mathbf{Bimod}_{\mathbf{Mod}(\mathcal{B})}(\mathcal{C}^{mop}\boxtimes\mathcal{C}).$$ More precisely, $\mathcal{C}^{mop}\boxtimes\mathcal{C}$ is equipped with the braided monoidal functor $$\mathcal{B}\xrightarrow{F}\mathcal{Z}(\mathcal{C})\hookrightarrow \mathcal{Z}(\mathcal{C}^{mop})\boxtimes\mathcal{Z}(\mathcal{C})\simeq\mathcal{Z}(\mathcal{C}^{mop}\boxtimes\mathcal{C}).$$ Then, as $\mathcal{Z}(\mathcal{C})$ and $\mathcal{C}^{mop}\boxtimes\mathcal{C}$ are Morita equivalent fusion 1-categories, $\mathbf{Mod}(\mathcal{Z}(\mathcal{C}))$ and $\mathbf{Mod}(\mathcal{C}^{mop}\boxtimes\mathcal{C})$ are equivalent finite semisimple 2-categories. Furthermore, this Morita equivalence of fusion 1-categories is compatible with the $\mathcal{B}$-central structures as the following diagram of braided monoidal functors commute $$\begin{tikzcd}
& \mathcal{B} \arrow[ld] \arrow[rd] &  \\
\mathcal{Z}(\mathcal{C}^{mop}\boxtimes\mathcal{C}) \arrow[rr, "\simeq"] & & \mathcal{Z}(\mathcal{Z}(\mathcal{C})).
\end{tikzcd}$$ But, as the $\mathcal{B}$-central structure on $\mathcal{C}^{mop}\boxtimes\mathcal{C}$ factors through $\mathcal{C}$, we have that $$\mathbf{Bimod}_{\mathbf{Mod}(\mathcal{B})}(\mathcal{C}^{mop}\boxtimes\mathcal{C})\simeq \mathbf{Bimod}(\mathcal{C}^{mop})\boxtimes \mathbf{Bimod}_{\mathbf{Mod}(\mathcal{B})}(\mathcal{C})$$ as fusion 2-categories. Finally, the fusion 2-categories $\mathbf{Bimod}(\mathcal{C}^{mop})$ and $\mathbf{Mod}(\mathcal{Z}(\mathcal{C}^{mop}))$ are equivalent (see example 5.3.7 of \cite{D8}). But, thanks to lemma 4.3 of \cite{D3}, the set of connected components of a 2-Deligne tensor product is the product of the set of connected components, so that $\mathbf{Bimod}_{\mathbf{Mod}(\mathcal{B})}(\mathcal{C})$ is indeed connected.
\end{proof}

\begin{Corollary}\label{cor:dualwittequivalent}
Let $\mathcal{B}$ be a braided fusion 1-category, and $\mathcal{C}$ be a $\mathcal{B}$-central fusion 1-category such that the braided monoidal functor $\mathcal{B}\rightarrow \mathcal{Z}(\mathcal{C})$ is fully faithful. Then, the dual to $\mathbf{Mod}(\mathcal{B})$ with respect to $\mathbf{Mod}(\mathcal{C})$ is given by $\mathbf{Mod}(\mathcal{Z}(\mathcal{C},\mathcal{B}))$.
\end{Corollary}

%% file: F2CMorita.tex
\section{Fusion 2-Categories up to Morita Equivalence}\label{sec:F2CmoduloMorita}

We analyze the Morita equivalence classes of fusion 2-categories. More precisely, we show that every fusion 2-category is Morita equivalent to the 2-Deligne tensor product of a strongly fusion 2-category and an invertible fusion 2-category. Further, we show that every strongly fusion 2-category is Morita equivalent to a connected fusion 2-category. In particular, this establishes that every fusion 2-category is Morita equivalent to a connected fusion 2-category. Throughout, we work over a fixed algebraically closed field $\mathds{k}$ of characteristic zero.

\subsection{Strongly Fusion and Invertible}

We begin by introducing invertible multifusion 2-categories, and we characterize them explicitly.

\begin{Definition}
A multifusion 2-category $\mathfrak{C}$ is invertible if $\mathscr{Z}(\mathfrak{C})\simeq \mathbf{2Vect}$.
\end{Definition}

\begin{Remark}
Let $\mathfrak{C}$ be an invertible multifusion 2-category. It follows from lemma \ref{lem:module2funcenter}, that the Morita equivalence class of $\mathfrak{C}\boxtimes \mathfrak{C}^{mop}$ is trivial. This justifies the name. In fact, thanks to theorem \ref{thm:MoritaInvarianceDrinfeldCenter} the latter gives an equivalent characterization of invertibility.
\end{Remark}

\begin{Lemma}\label{lem:invertiblecharacterization}
A fusion 2-category is invertible if and only if it is equivalent to $\mathbf{Mod}(\mathcal{B})$ with $\mathcal{B}$ a non-degenerate braided fusion 1-category.
\end{Lemma}
\begin{proof}
Let $\mathfrak{C}$ be an invertible fusion 2-category. This means that the dual to $\mathfrak{C}\boxtimes\mathfrak{C}^{mop}$ with respect to $\mathfrak{C}$ is $\mathbf{2Vect}$. Thanks to corollary 5.4.4 of \cite{D8}, this implies that the dual to $\mathbf{2Vect}$ with respect to $\mathfrak{C}$ is $\mathfrak{C}\boxtimes\mathfrak{C}^{mop}$. Thus, by combining proposition 5.1.6 and corollary 5.2.5 of \cite{D8}, we get that $\mathfrak{C}$ is indecomposable as a finite semisimple 2-category, i.e.\ it is connected. Let $\mathcal{B}$ be a braided fusion 1-category such that $\mathbf{Mod}(\mathcal{B})\simeq \mathfrak{C}$ as fusion 2-category. As $\mathscr{Z}(\mathbf{Mod}(\mathcal{B}))\simeq\mathbf{2Vect}$, it follows from lemma 2.16 of \cite{JFR}, or corollary \ref{cor:omegacenter} above, that $\mathcal{Z}_{(2)}(\mathcal{B})\simeq \mathbf{Vect}$, so that $\mathcal{B}$ is non-degenerate.
\end{proof}

Before stating the first theorem of this section, we will establish a technical lemma. In order to do so, it is convenient to generalize the terminology of definition \ref{def:stronglyfusion}.

\begin{Definition}
We say that a fusion 2-category $\mathfrak{C}$ is bosonic if the symmetric fusion 1-category $\mathcal{Z}_{(2)}(\Omega\mathfrak{C})$ is Tannakian. Otherwise, we say that $\mathfrak{C}$ is fermionic.
\end{Definition}

\begin{Lemma}\label{lem:dualdisconnected}
Every bosonic fusion 2-category $\mathfrak{C}$ is Morita equivalent to a fusion 2-category $\mathfrak{D}$ with $\Omega\mathfrak{D}$ a non-degenerate braided fusion 1-category. Every fermionic fusion 2-category $\mathfrak{C}$ is Morita equivalent to a fusion 2-category $\mathfrak{D}$ with $\Omega\mathfrak{D}$ a slightly-degenerate braided fusion 1-category, i.e. $\mathcal{Z}_{(2)}(\Omega\mathfrak{D})=\mathbf{SVect}$.
\end{Lemma}
\begin{proof}
Let $\mathfrak{C}$ be a bosonic fusion 1-category. We write $\mathcal{E}:=\mathcal{Z}_{(2)}(\Omega\mathfrak{C})$. Thanks to Deligne's theorem (see \cite{De}), there exists a symmetric monoidal functor $\mathcal{E}\rightarrow\mathbf{Vect}$. We define a braided fusion 1-category $\mathcal{D}:=\Omega\mathfrak{C}\boxtimes_{\mathcal{E}}\mathbf{Vect}$. By proposition 4.30 of \cite{DGNO}, $\mathcal{D}$ is non-degenerate. Further, we view $\mathcal{D}$ as a connected rigid algebra in $\mathbf{Mod}(\Omega\mathfrak{C})\simeq \mathfrak{C}^0$ using the canonical braided monoidal functor $\Omega\mathfrak{C}\rightarrow \mathcal{D}\rightarrow \mathcal{Z}(\mathcal{D})$. It follows from example \ref{ex:algebrasModB} that $\mathcal{D}$ is a separable algebra. Thus, by theorem 5.4.3 of \cite{D8}, the multifusion 2-category $\mathbf{Bimod}_{\mathfrak{C}}(\mathcal{D})$ is Morita equivalent to $\mathfrak{C}$. Further, as $\mathcal{D}$ is connected, it is necessarily indecomposable, so that it follows from section 5.2 of \cite{D8} that $\mathbf{Bimod}_{\mathfrak{C}}(\mathcal{D})$ is in fact a fusion 2-category, and we set $\mathfrak{D}:=\mathbf{Bimod}_{\mathfrak{C}}(\mathcal{D})$. It remains to analyze $\mathfrak{D}^0$. As $\mathcal{D}$ is an algebra in $\mathfrak{C}^0$, corollary 2.3.6 of \cite{D2} shows that the underlying object of any simple $\mathcal{D}$-$\mathcal{D}$-bimodule in $\mathfrak{C}$ is contained in a single connected component of $\mathfrak{C}$. In particular, we have an equivalence $$\mathfrak{D}^0\simeq (\mathbf{Bimod}_{\mathfrak{C}^0}(\mathcal{D}))^0\simeq (\mathbf{Bimod}_{\mathbf{Mod}(\Omega\mathfrak{C})}(\mathcal{D}))^0$$ of fusion 2-categories. Moreover, it was shown in lemma \ref{lem:dualconnectedcomponent} above that $$(\mathbf{Bimod}_{\mathbf{Mod}(\Omega\mathfrak{C})}(\mathcal{D}))^0\simeq \mathbf{Mod}(\mathcal{Z}(\mathcal{D},\Omega\mathfrak{C}))$$ as fusion 2-categories. Finally, as $\mathcal{D}$ is non-degenerate, we have that $\mathcal{Z}(\mathcal{D})\simeq \mathcal{D}\boxtimes\mathcal{D}^{rev}$ by proposition 3.7 of \cite{DGNO}. Further, the canonical braided monoidal functor $\Omega\mathfrak{C}\rightarrow \mathcal{D}$ is surjective (in the sense of definition 2.1 of \cite{DGNO}), so that $\mathcal{Z}(\mathcal{D},\Omega\mathfrak{C})\simeq \mathcal{D}^{rev}$ as braided fusion 1-categories. This concludes the proof in this case.

If $\mathfrak{C}$ is fermionic, then it follows from \cite{De} that there exists a dominant symmetric monoidal functor $\mathcal{E}\rightarrow\mathbf{SVect}$. We then define a braided fusion 1-category $\mathcal{D}:=\Omega\mathfrak{C}\boxtimes_{\mathcal{E}}\mathbf{SVect}$. By proposition 4.30 of \cite{DGNO}, $\mathcal{D}$ is slightly-degenerate. Then, the argument used above shows that $\mathfrak{D}:=\mathbf{Bimod}_{\mathbf{Mod}(\Omega\mathfrak{C})}(\mathcal{D})$ is a fusion 2-category with $\mathfrak{D}^0\simeq \mathbf{Mod}(\mathcal{Z}(\mathcal{D},\Omega\mathfrak{C}))$ as fusion 2-categories. It follows from proposition 4.3 of \cite{DNO} that $\mathcal{Z}(\mathcal{D},\Omega\mathfrak{C})$ is slightly-degenerate as desired, thereby finishing the proof.
\end{proof}

\begin{Theorem}\label{thm:stronglyfusioninvertible}
Every bosonic, respectively fermionic, fusion 2-category is Morita equivalent to the 2-Deligne tensor product of a bosonic, respectively fermionic, strongly fusion 2-category and an invertible fusion 2-category.
\end{Theorem}
\begin{proof}
Let $\mathfrak{C}$ be a fusion 2-category, and write $\mathcal{E}$ for the symmetric fusion 1-category $\mathcal{Z}_{(2)}(\Omega\mathfrak{C})$. We treat the cases of bosonic and fermionic fusion 2-categories separately.

\noindent \textbf{Case 1:} We assume that the fusion 2-category $\mathfrak{C}$ is bosonic. In that case, it follows from lemma \ref{lem:dualdisconnected} that there exists a fusion 2-category $\mathfrak{D}$ that is Morita equivalent to $\mathfrak{C}$, and such that $\mathcal{B}:=\Omega\mathfrak{D}$ is a non-degenerate braided fusion 1-category. Let us now define $\mathfrak{F}:= \mathfrak{D}\boxtimes \mathbf{Mod}(\mathcal{B}^{rev})$, so that $\Omega\mathfrak{F}\simeq \mathcal{B}\boxtimes \mathcal{B}^{rev}$ by proposition 5.1 of \cite{D3}. In particular, equipping $\mathcal{B}$ with the canonical braided monoidal functor $\mathcal{B}\boxtimes \mathcal{B}^{rev}\rightarrow \mathcal{Z}(\mathcal{B})$, we can view $\mathcal{B}$ as a separable algebra in the fusion 2-category $\mathbf{Mod}(\mathcal{B}\boxtimes \mathcal{B}^{rev})\simeq\mathfrak{F}^0$. Let us write $\mathfrak{G}$ for the fusion 2-category given by $\mathbf{Bimod}_{\mathfrak{F}}(\mathcal{B})$. By example 5.3.7 of \cite{D8}, see also lemma \ref{lem:dualconnectedstrong} above, we find that $$\mathfrak{G}^0= \mathbf{Bimod}_{\mathfrak{F}^0}(\mathcal{B})\simeq \mathbf{2Vect}.$$ This shows that $\mathfrak{G}$ is a bosonic strongly fusion 2-category.

Finally, it follows from lemma \ref{lem:bimodule2Deligne} that Morita equivalence is compatible with the 2-Deligne tensor product. In particular, as $\mathbf{2Vect}$ is Morita equivalent to $\mathbf{Mod}(\mathcal{B}^{rev})\boxtimes\mathbf{Mod}(\mathcal{B})$, we have that $\mathfrak{C}$ is Morita equivalent to the fusion 2-category $\mathfrak{C}\boxtimes\mathbf{Mod}(\mathcal{B}^{rev})\boxtimes\mathbf{Mod}(\mathcal{B})$.
Furthermore, the above discussion shows that $\mathfrak{C}\boxtimes\mathbf{Mod}(\mathcal{B}^{rev})\boxtimes\mathbf{Mod}(\mathcal{B})$ is Morita equivalent to $\mathfrak{G}\boxtimes \mathbf{Mod}(\mathcal{B})$. But, Morita equivalence is an equivalence relation by theorem \ref{thm:MoritaF2C}, so that we get the desired result in this case.

\noindent \textbf{Case 2:} We assume that the fusion 2-category $\mathfrak{C}$ is fermionic. In that case, it follows from lemma \ref{lem:dualdisconnected} that there exists a fusion 2-category $\mathfrak{D}$ that is Morita equivalent to $\mathfrak{C}$, and such that $\mathcal{B}:=\Omega\mathfrak{D}$ is a slightly degenerate braided fusion 1-category. Now, thank to the main theorem of \cite{JFR}, there exists a non-degenerate braided fusion 1-category $\mathcal{C}$ together with a braided embedding $\mathcal{B}\subseteq \mathcal{C}$, such that the centralizer of $\mathcal{B}$ in $\mathcal{C}$ is exactly $\mathbf{SVect}$. Let us write $\mathfrak{E}$ for the fusion 2-category $\mathbf{Bimod}_{\mathfrak{D}}(\mathcal{C})$. By lemma \ref{lem:dualconnectedstrong} above, we find that $$\mathfrak{E}^0= \mathbf{Bimod}_{\mathfrak{D}^0}(\mathcal{C})\simeq \mathbf{Mod}(\mathbf{SVect}\boxtimes \mathcal{C}^{rev}).$$ Namely, as $\mathcal{C}$ is a non-degenerate braided fusion 1-category, we have that $\mathcal{Z}(\mathcal{C})\simeq \mathcal{C}\boxtimes\mathcal{C}^{rev}$ by proposition 3.7 of \cite{DGNO}, so that $\mathcal{Z}(\mathcal{C},\mathcal{B})\simeq \mathbf{SVect}\boxtimes \mathcal{C}^{rev}$ as braided fusion 1-categories. Then, we consider the fusion 2-category $\mathfrak{F}:=\mathfrak{E}\boxtimes \mathbf{Mod}(\mathcal{C})$. Using the same argument as the one used in the proof of the first case, we find that $\mathfrak{G}:=\mathbf{Bimod}_{\mathfrak{F}}(\mathcal{C})$ is a fermionic strongly fusion 2-category.

Finally, as $\mathbf{2Vect}$ is Morita equivalent to $\mathbf{Mod}(\mathcal{C})\boxtimes\mathbf{Mod}(\mathcal{C}^{rev})$, we have that $\mathfrak{C}$ is Morita equivalent to the fusion 2-category $\mathfrak{C}\boxtimes\mathbf{Mod}(\mathcal{C})\boxtimes\mathbf{Mod}(\mathcal{C}^{rev})$.
Furthermore, the above discussion implies that $\mathfrak{C}\boxtimes\mathbf{Mod}(\mathcal{C})\boxtimes\mathbf{Mod}(\mathcal{C}^{rev})$ is Morita equivalent to $\mathfrak{G}\boxtimes \mathbf{Mod}(\mathcal{C}^{rev})$. This concludes the proof.
\end{proof}

\begin{Remark}\label{rem:decompositionnonunique}
We emphasize that the invertible fusion 2-category supplied by theorem \ref{thm:stronglyfusioninvertible} is not unique. Namely, if $\mathcal{C}$ is any minimal non-degenerate extension of $\mathbf{SVect}$, then it follows from theorem \ref{thm:WittEquivalence} that $\mathbf{2SVect}\boxtimes\mathbf{Mod}(\mathcal{C})$ is Morita equivalent to $\mathbf{2SVect}$. On the one hand, in the fermionic case, the strongly fusion 2-categories is not quite unique either, as will be explained in \cite{JF3}. On the other hand, we will show in proposition \ref{prop:bosonicWittgroups} below that, in the bosonic case, the strongly fusion 2-category supplied by theorem \ref{thm:stronglyfusioninvertible} is unique up to monoidal equivalence, and that the invertible fusion 2-category is unique up to Morita equivalence.
\end{Remark}

\subsection{Connected}

We say that an algebra $A$ in a fusion 2-category is called strongly connected if its unit $i:I\rightarrow A$ is the inclusion of a simple summand. We begin by proving the following lemma.

\begin{Lemma}\label{lem:criterionbimoduleconnected}
Let $A$ be a strongly connected separable algebra in the fusion 2-category $\mathfrak{C}$. If the underlying object of $A$ has a summand in every connected component of $\mathfrak{C}$, then $\mathbf{Bimod}_{\mathfrak{C}}(A)$ is a connected fusion 2-category.
\end{Lemma}
\begin{proof}
As $A$ is connected, it is a fortiori indecomposable, which ensures that $\mathbf{Bimod}_{\mathfrak{C}}(A)$ is a fusion 2-category by corollary 5.2.5 of \cite{D8}. Let $P$ be any simple $A$-$A$-bimodule. There exists a simple object $C$ of $\mathfrak{C}$ and a non-zero 1-morphism $A\Box C\Box A\rightarrow P$ of $A$-$A$-bimodules. For instance, pick any summand of $P$ viewed as an object of $\mathfrak{C}$, and consider the induced map of $A$-$A$-bimodules. We claim that $A\Box C\Box A$ is a simple $A$-$A$-bimodule. Namely, under the adjunction $$Hom_{A\textrm{-}A}(A\Box C\Box A,A\Box C\Box A)\simeq Hom_{\mathfrak{C}}(C,A\Box C\Box A),$$ the identity $A$-$A$-bimodule 1-morphism on $P$ corresponds to the 1-morphism $i\Box Id_C\Box i$ in $\mathfrak{C}$. But, thanks to our assumptions, the 1-morphism $i:I\hookrightarrow A$ is the inclusion of the monoidal unit. It therefore follows that $i\Box Id_C\Box i$ is a simple 1-morphism in $\mathfrak{C}$, which establishes that $A\Box C\Box A$ is a simple $A$-$A$-bimodule. Now, $A$ has a summand in every connected component of $\mathfrak{C}$, so that there exists a non-zero 1-morphism $f:C\rightarrow A$. Further, as $A$ is connected, the composite $A$-$A$-bimodule 1-morphism $A\Box C\Box A\rightarrow A\Box A\Box A\rightarrow A\Box A$ is non-zero. Thence, every simple object of $\mathbf{Bimod}_{\mathfrak{C}}(A)$ is in the connected component of the simple object $A\Box A$, concluding the proof of the lemma.
\end{proof}

We now establish the second main theorem of this section.

\begin{Theorem}\label{thm:Moritaconnected}
Every fusion 2-category is Morita equivalent to a connected fusion 2-category.
\end{Theorem}
\begin{proof}
By theorem \ref{thm:stronglyfusioninvertible}, every fusion 2-category is Morita equivalent to the 2-Deligne tensor product of a strongly fusion 2-category and an invertible fusion 2-category. It is therefore enough to establish that every strongly fusion 2-category is Morita equivalent to a connected fusion 2-category.

We fix $\mathfrak{C}$ a strongly fusion 2-category, and write $\mathfrak{C}^{\times}$ for the monoidal sub-2-category of $\mathfrak{C}$ on the invertible objects and the invertible morphisms. As $\mathfrak{C}$ is a monoidal $\mathds{k}$-linear 2-category, there is an equivalence of monoidal 2-categories $\mathfrak{C}^{\times}\simeq \mathcal{G}\times^{\pi}\mathrm{B}^2\mathds{k}^{\times}$ for some finite 2-group $\mathcal{G}$ and 4-cocycle $\pi$ for $\mathcal{G}$ with coefficients in $\mathds{k}^{\times}$. We consider the fusion 2-category $\mathbf{2Vect}_{\mathcal{G}}^{\pi}$ from example \ref{ex:2groupgraded2vectorspaces}. Then, the monoidal inclusion $\mathfrak{C}^{\times}\hookrightarrow \mathfrak{C}$ induces a monoidal 2-functor $\mathbf{J}:\mathbf{2Vect}_{\mathcal{G}}^{\pi}\rightarrow \mathfrak{C}$. This follows from the fact that $\mathbf{2Vect}_{\mathcal{G}}^{\pi}$ is constructed from $\mathfrak{C}^{\times}$ through linearization, local Cauchy completion, and Cauchy completion, which are all operations satisfying sufficient universal properties. Moreover, because $\mathfrak{C}$ is a strongly fusion 2-category, the main results of \cite{JFY} yield that every simple object of $\mathfrak{C}$ in invertible. As a consequence, we find that $\mathbf{J}$ is essentially surjective on objects.

Now, if follows lemma \ref{lem:structurelinear3groups} below that there exists a finite group $H$ and a monoidal 2-functor $H\times\mathrm{B}^2\mathds{k}^{\times}\rightarrow \mathcal{G}\times^{\pi}\mathrm{B}^2\mathds{k}^{\times}\simeq\mathfrak{C}^{\times}$, which is essentially surjective on objects. After linearization, local Cauchy completion, and Cauchy completion, this yields a monoidal 2-functor $\mathbf{K}:\mathbf{2Vect}_H\rightarrow \mathbf{2Vect}_{\mathcal{G}}^{\pi}$ between fusion 2-categories such that the composite $\mathbf{J}\circ\mathbf{K}$ is essentially surjective on objects. The fusion 1-category $\mathbf{Vect}_H$ equipped with the non-trivial grading by $H$ defines a strongly connected rigid algebra in $\mathbf{2Vect}_H$, which is separable thanks to corollary 3.3.7 of \cite{D7}. We write $A$ for the separable algebra $\mathbf{J}(\mathbf{K}(\mathbf{Vect}_H))$ in $\mathfrak{C}$. We note that $A$ is strongly connected as $\mathbf{J}\circ\mathbf{K}$ is a monoidal 2-functor between fusion 2-categories. By construction, we also have that the underlying object of $A$ contains every simple object of $\mathfrak{C}$ as a direct summand. Thus, the hypotheses of lemma \ref{lem:criterionbimoduleconnected} are satisfied, and we get that $\mathbf{Bimod}_{\mathfrak{C}}(A)$ is a connected fusion 2-category.
\end{proof}

We will need the following technical lemma generalizing theorem 1 of \cite{Op} to cohomology with twisted coefficients.

\begin{Lemma}\label{lem:pullbackvanishing}
Let $G$ be a finite group, and $A$ a torsion abelian group equipped with a $G$ action. For $n\geq 2$, given any $n$-cocycle $\omega$ for $G$ with coefficients in $A$, there exists a finite group $\widetilde{G}$ together with a surjective group homomorphism $p:\widetilde{G}\rightarrow G$ such that the pullback $p^*\omega$ is a coboundary.
\end{Lemma}
\begin{proof}
We begin by recording the following fact, which is an immediate consequence of corollary 6.2.23 of \cite{Wei}, and which we will use repeatedly:\ For any free group $F$, abelian group $B$ with an action by $F$, and $m\geq 2$, we have $H^m(F;B)=0$. Let us now pick a surjective group homomorphism $f:F\twoheadrightarrow G$ with $F$ a finitely generated free group. In particular, it is clear that $f^*\omega$ is a coboundary. For our purposes, it is useful to recast this observation through a Lyndon-Hochschild-Serre spectral sequence. More precisely, let us write $N$ for the kernel of $f$, and note that, thanks to the Nielsen-Schreier theorem, $N$ is again a finitely generated free group, and that it acts trivially on $A$. We can consider the Lyndon-Hochschild-Serre spectral sequence \cite{HS} $$E_2^{p,q}=H^p(G;H^q(N;A))\Rightarrow H^{p+q}(F;A)$$ associated to the short exact sequence of groups $N\hookrightarrow F\twoheadrightarrow G$. By the previous observations, the $E_2$ page of this spectral sequence is zero unless $q=0,1$, and its $E_{\infty}$ page is zero unless $p+q\leq 1$. In particular, $\omega$ viewed as an element of $E^{0,n}_2=H^n(G;A)$ is in the image of the $d_2$ differential $$d_2:H^1(N;H^{n-2}(G;A))=Hom(N,H^{n-2}(G;A))\rightarrow H^0(N;H^n(G;A))=H^n(G;A).$$ Let us write $\psi:N\rightarrow H^{n-2}(G;A)$ for a homomorphism representing a class in $H^1(N;H^{n-2}(G;A))$ such that $d_2(\psi) = \omega$. Now, if $k$ denotes the exponent of the image of $\psi$ in the torsion abelian group $H^{n-2}(G;A)$, we have that $\psi$ factors through the quotient $N/K$ with $K$ the subgroup of $N$ generated by commutators and $k$-th powers. Moreover, the subgroup $K$ of $N$ is characteristic, i.e.\ it is preserved by any automorphism of $N$, so that $K$ is also a normal subgroup of $F$. In particular, we may consider the quotient group $\widetilde{G}:=F/K$, which is finite as $N/K$ is finite. Furthermore, there is a surjective group homomorphism $p:\widetilde{G}\rightarrow G$. We can therefore consider the following homomorphism of short exact sequences of groups: $$\begin{tikzcd}[sep=small]
N \arrow[r, hook] \arrow[d] & F \arrow[r, two heads] \arrow[d]   & G \arrow[d] \\
N/K \arrow[r, hook]         & F/K=\widetilde{G} \arrow[r, two heads] & G.          
\end{tikzcd}$$ By naturality of the Lydon-Hochschild-Serre spectral sequence, there is a map of spectral sequence $\widetilde{E}^{p,q}_2\rightarrow E^{p,q}_2$. Moreover, as $\psi$ factors through $N/K$, it can be lifted to $\widetilde{\psi}$ in $\widetilde{E}^{1,n-2}_2=Hom(N/K;H^{n-2}(G;A))$, and it follows that $\widetilde{d}_2(\widetilde{\psi})=\omega$. Said differently, the class represented by $\omega$ in $\widetilde{E}^{0,n}_2$ is killed by $\widetilde{d}_2$, and therefore does not survive to $\widetilde{E}^{0,n}_{\infty}$. This is exactly saying that the pullback $p^*\omega$ is a coboundary as the subgroup $\widetilde{E}^{0,n}_{\infty}$ of $H^n(\widetilde{G};A)$ is the image of the pullback map $p^*$.
\end{proof}

\begin{Lemma}\label{lem:structurelinear3groups}
Let $\mathcal{G}$ be any finite 2-group, and $\pi$ any 4-cocycle for $\mathcal{G}$ with coefficients in $\mathds{k}^{\times}$. Then, there exists a finite group $H$ and an essentially surjective monoidal functor $K:H\rightarrow \mathcal{G}$ such that the pullback of $\pi$ along $K$ is a coboundary.
\end{Lemma}
\begin{proof}
Let us write $G:= \pi_1(\mathcal{G})$ for the finite group of equivalence classes of objects of $\mathcal{G}$, and $A:=\pi_2(\mathcal{G})$ for the finite abelian group of automorphisms of the monoidal unit of $\mathcal{G}$. Further, $G$ acts on $A$ by conjugation. It was shown in \cite{Sin} that the finite 2-group $\mathcal{G}$ is uniquely determined by the above data together with $\omega$, a 3-cocycle for $G$ with coefficients in $A$. Thanks to lemma \ref{lem:pullbackvanishing}, there exists a finite group $\widetilde{G}$, and a surjective group homomorphism $p:\widetilde{G}\twoheadrightarrow G$ such that $p^*\omega$ is a coboundary. In particular, the data of $p$ together with a cochain trivializing $p^*\omega$ induces an essentially surjective monoidal functor $p:\widetilde{G}\rightarrow \mathcal{G}$. Then, $p^*\pi$, the pullback of $\pi$ along $p$, is a 4-cocycle for $\widetilde{G}$ with coefficients in $\mathds{k}^{\times}$. By lemma \ref{lem:pullbackvanishing} above, there exists a finite group $H$ and a surjective group homomorphism $q:H\twoheadrightarrow \widetilde{G}$ such that $q^*p^*\pi$ is a coboundary. Taking $K:=p\circ q$ concludes the proof of the lemma.
\end{proof}

\begin{Remark}
It follows from theorem \ref{thm:Moritaconnected} and corollary \ref{cor:Moritaclassesconnected} that Morita equivalence classes of fusion 2-categories correspond exactly to pairs consisting of a symmetric fusion 1-category $\mathcal{E}$ and a class in $\widehat{\mathcal{W}}(\mathcal{E})$, the quotient of the Witt group $\mathcal{W}(\mathcal{E})$ under the action of $Aut^{br}(\mathcal{E})$. Furthermore, it was shown in theorem 5.5 of \cite{DNO} that every class in $\mathcal{W}(\mathcal{E})$ is represented by a unique completely $\mathcal{E}$-anisotropic braided fusion 1-category equipped with an identification of its symmetric center with $\mathcal{E}$. Thus, every class in $\widehat{\mathcal{W}}(\mathcal{E})$ is represented by a unique completely $\mathcal{E}$-anisotropic braided fusion 1-category. Fixing such a braided fusion 1-category $\mathcal{B}$, the fusion 2-categories that are Morita equivalent to $\mathbf{Mod}(\mathcal{B})$ are classified by indecomposable finite semisimple left $\mathbf{Mod}(\mathcal{B})$-module 2-categories. By theorem 5.1.2 of \cite{D8}, the later are exactly given by the Morita equivalence classes of $\mathcal{B}$-central fusion 1-categories.
\end{Remark}

Our last theorem has the following interesting consequence.

\begin{Corollary}
The property of being bosonic, respectively fermionic, is preserved under Morita equivalence of fusion 2-categories.
\end{Corollary}
\begin{proof}
Let $\mathcal{B}$ be a braided fusion 1-category, and $\mathcal{C}$ be a $\mathcal{B}$-central fusion 1-category. We will show that the symmetric center of $\mathcal{Z}(\mathcal{C},\mathcal{B})$ is Tannakian if and only if the symmetric center of $\mathcal{B}$ is Tannakian. By virtue of theorem \ref{thm:Moritaconnected} and lemma \ref{lem:dualconnectedcomponent}, this yields the statement of the corollary. Let us write $\mathcal{A}$ for the image of $\mathcal{B}$ in $\mathcal{Z}(\mathcal{C})$, which is a braided fusion 1-category, and note that $\mathcal{Z}(\mathcal{C},\mathcal{B})\simeq \mathcal{Z}(\mathcal{C},\mathcal{A})$ as braided fusion 1-categories. It follows from corollary 3.2.4 of \cite{DMNO} that the symmetric center of $\mathcal{A}$ is Tannakian if and only the symmetric center of $\mathcal{B}$ is Tannakian. Finally, proposition 4.3 of \cite{DNO} establishes that $\mathcal{Z}_{(2)}(\mathcal{A})\simeq \mathcal{Z}_{(2)}(\mathcal{Z}(\mathcal{C},\mathcal{A}))$ as symmetric fusion 1-categories. This concludes the proof of the claim.
\end{proof}

%% file: Separable.tex
\section{Separability}\label{sec:separable}

We begin by proving that every rigid algebra in a fusion 2-category is separable. As a consequence, we find that every multifusion 2-category is separable. Following \cite{JF}, we comment on the relation between the separability of multifusion 2-categories and the fact that multifusion 2-categories are fully dualizable objects of an appropriate symmetric monoidal 4-category. We also introduce a notion of dimension for fusion 2-categories, and check that it is always non-vanishing. Finally, we show that the Drinfeld center of any fusion 2-category is a finite semisimple 2-category. Throughout, we work over a fixed algebraically closed field $\mathds{k}$ of characteristic zero.

\subsection{Rigid Algebras}

Thanks to theorem \ref{thm:Moritaconnected}, we have an extremely good control over the Morita equivalence classes of fusion 2-categories. We use this to prove the following deep result on the structure of rigid algebras in fusion 2-categories, which internalizes corollary 2.6.8 of \cite{DSPS13}, and positively answers a question raised in \cite{JFR}.

\begin{Theorem}\label{thm:algebrarigidseparable}
Every rigid algebra in a fusion 2-category is separable.
\end{Theorem}
\begin{proof}
Let $\mathfrak{C}$ be any fusion 2-category. Thanks to theorems \ref{thm:MoritaF2C} and \ref{thm:Moritaconnected}, there exists a connected fusion 2-category $\mathfrak{D}$, a separable algebra $A$ in $\mathfrak{D}$, and a monoidal equivalence $\mathfrak{C}\simeq\mathbf{Bimod}_{\mathfrak{D}}(A)$ of fusion 2-categories. Under the above equivalence, it follows from corollary 3.2.12 of \cite{D8} that any rigid algebra in $\mathfrak{C}$ corresponds to a rigid algebra $B$ in $\mathfrak{D}$ equipped with an algebra 1-homomorphism $A\rightarrow B$. But, as $\mathfrak{D}$ is connected, example \ref{ex:algebrasModB} implies that $B$ is a separable algebra in $\mathfrak{D}$. Thence, corollary 3.2.12 of \cite{D8} demonstrates that the corresponding rigid algebra in $\mathfrak{C}$ is separable as desired.
\end{proof}

\begin{Corollary}\label{cor:multialgebrarigidseparable}
Every rigid algebra in a multifusion 2-category is separable.
\end{Corollary}
\begin{proof}
Every multifusion 2-categories can be split into a direct sum of finitely many indecomposable multifusion 2-categories by lemma 5.2.10 of \cite{D8}. In particular, this induces a decomposition of any algebra into a finite direct sum of algebras contained in exactly one indecomposable multifusion 2-category. It is therefore enough to show that any rigid algebra in an indecomposable multifusion 2-category is separable. In order to see this, note that every indecomposable multifusion 2-category is Morita equivalent to a fusion 2-category as can be seen from remark 5.4.2 of \cite{D8}. The result then follows from the proof of theorem \ref{thm:algebrarigidseparable} above.
\end{proof}

By combining the last corollary with theorem 5.3.4 of \cite{D4}, we get the following result.

\begin{Corollary}
Every finite semisimple module 2-category over a multifusion 2-category is separable.
\end{Corollary}

\begin{Remark}
Provided we are working over an algebraically closed field of characteristic zero, we find a posteriori that all the results of section 5 of \cite{D8} hold even with the separability hypothesis removed. In particular, for any multifusion 2-category $\mathfrak{C}$, there is an equivalence between the Morita 3-category of rigid algebras in $\mathfrak{C}$ and the 3-category of finite semisimple left $\mathfrak{C}$-module 2-categories. This affirmatively answers the question raised in remark 5.3.9 of \cite{D4}.
\end{Remark}

\begin{Remark}
Let $\mathfrak{C}$ be a multifusion 2-category. As a consequence of the above results, one can show that the relative 2-Deligne tensor product of any finite semisimple $\mathfrak{C}$-module 2-categories exists. In particular, it then follows from \cite{JFS} that multifusion 2-categories, finite semisimple bimodule 2-categories, and bimodule morphisms form a symmetric monoidal 4-category, which we denote by $\mathbf{F2C}$. Let us also consider $\mathbf{BrFus}$, the Morita 4-category of braided multifusion 1-categories studied in \cite{BJS}. We expect that the assignment $$\mathbf{Mod}(-):\mathbf{BrFus}\rightarrow \mathbf{F2C}$$ given by sending a braided multifusion 1-category $\mathcal{B}$ to the associated multifusion 2-category $\mathbf{Mod}(\mathcal{B})$ extends to a symmetric monoidal 4-functor (see the second proof of theorem 1 of \cite{JF} for a related result). Theorem \ref{thm:Moritaconnected} proves that the 4-functor $\mathbf{Mod}$ is essentially surjective on objects. Moreover, it follows from theorem 5.1.2 of \cite{D8} and lemma 2.2.6 of \cite{D4} that $\mathbf{Mod}$ induces equivalence between $Hom$-3-categories. Thus, provided it exists, the symmetric monoidal 4-functor $\mathbf{Mod}$ is an equivalence. As braided multifusion 1-categories are fully dualizable objects of $\mathbf{BrFus}$ thanks to \cite{BJS}, it would then follow that multifusion 2-categories are fully dualizable objects of $\mathbf{F2C}$. We give another approach to establishing the full dualizability of multifusion 2-categories in remark \ref{rem:4dualizabilityhighercondensations} below.
\end{Remark}

\subsection{Fusion 2-Categories}

We work with a fixed multifusion 2-category $\mathfrak{C}$. The following definition categorifies definition 2.5.8 of \cite{DSPS13}.

\begin{Definition}\label{def:separableF2C}
The multifusion 2-category $\mathfrak{C}$ is called separable if it is separable as a finite semisimple left $\mathfrak{C}\boxtimes\mathfrak{C}^{mop}$-module 2-category.
\end{Definition}

We now associate a canonical rigid algebra to any multifusion 2-category. This construction should be compared with that of section 2.6 of \cite{DSPS13}, in which a canonical Frobenius algebra is associated to any fusion 1-category.

\begin{Construction}\label{con:canonicalalgebra}
As $\mathfrak{C}\boxtimes\mathfrak{C}^{mop}$ is a multifusion 2-category acting on the finite semisimple 2-category $\mathfrak{C}$, it follows from theorem 4.2.2 of \cite{D4} that $\mathfrak{C}$ can be given the structure of a $\mathfrak{C}\boxtimes\mathfrak{C}^{mop}$-enriched 2-category. More precisely, for any $C,D$ in $\mathfrak{C}$, the $Hom$-object $\underline{Hom}(C,D)$ in $\mathfrak{C}\boxtimes\mathfrak{C}^{mop}$ is characterized by the adjoint 2-natural equivalence $$Hom_{\mathfrak{C}}(H\Box C, D)\simeq Hom_{\mathfrak{C}\boxtimes\mathfrak{C}^{mop}}(H, \underline{Hom}(C,D)),$$ for every $H$ in $\mathfrak{C}\boxtimes\mathfrak{C}^{mop}$. It was established in proposition 4.1.1 of \cite{D4} that such an object $\underline{Hom}(C,D)$ of $\mathfrak{C}\boxtimes\mathfrak{C}^{mop}$ always exists. We write $\mathcal{R}_{\mathfrak{C}}$ for the algebra $\underline{End}(I)$ in $\mathfrak{C}\boxtimes\mathfrak{C}^{mop}$, and observe that $\mathcal{R}_{\mathfrak{C}}$ is rigid by theorem 5.2.7 of \cite{D4}. Moreover, the monoidal unit $I$ of $\mathfrak{C}$ generates $\mathfrak{C}$ under the action of $\mathfrak{C}\boxtimes\mathfrak{C}^{mop}$ in the sense of definition 5.3.1 of \cite{D4}. Thus, by theorem 5.3.4 of \cite{D4}, there is an equivalence $$\mathfrak{C}\simeq \mathbf{Mod}_{\mathfrak{C}\boxtimes\mathfrak{C}^{mop}}(\mathcal{R}_{\mathfrak{C}})$$ of left $\mathfrak{C}\boxtimes\mathfrak{C}^{mop}$-module 2-categories. For later use, let us also record that, if we assume that $\mathfrak{C}$ is a fusion 2-category, then $\mathcal{R}_{\mathfrak{C}}$ is connected. Namely, its unit 1-morphism corresponds to the identity 1-morphism $I\rightarrow I$ under the adjunction above.
\end{Construction}

Thanks to theorem \ref{thm:algebrarigidseparable} and corollary \ref{cor:multialgebrarigidseparable}, the rigid algebra $\mathcal{R}_{\mathfrak{C}}$ is in fact separable. This yields the following result, which categorifies corollary 2.6.8 of \cite{DSPS13}.

\begin{Corollary}\label{cor:separable}
Every multifusion 2-category is separable.
\end{Corollary}

\begin{Remark}\label{rem:4dualizabilityhighercondensations}
We now comment on the relation between 4-dualizability and separability for multifusion 2-categories. We begin by recalling the setup considered in section 2 of \cite{JF}. Let us write $\mathbf{Cau2Cat}_{\mathds{k}}$ for the linear 3-category of Cauchy complete $\mathds{k}$-linear 2-categories, and note that the completed tensor product $\widehat{\otimes}$ endows this 3-category with a symmetric monoidal structure. We will assume that $\mathbf{Cau2Cat}_{\mathds{k}}$ is closed under colimits and that $\widehat{\otimes}$ preserves them. Then, thanks to the results of section 8 of \cite{JFS}, we can consider the symmetric monoidal 4-category $\mathbf{Mor}_1(\mathbf{Cau2Cat}_{\mathds{k}})$ of Cauchy complete monoidal 2-categories, Cauchy complete bimodule 2-categories, and bimodule morphisms. Theorem 1 of \cite{JF} then sketches a proof of the fact that a multifusion 2-category $\mathfrak{C}$ is a fully dualizable object, that is a 4-dualizable object, of $\mathbf{Mor}_1(\mathbf{Cau2Cat}_{\mathds{k}})$ if and only if it is 2-dualizable. More precisely, let $T:\mathfrak{C}\boxtimes\mathfrak{C}^{mop}\rightarrow \mathfrak{C}$ denote the canonical left $\mathfrak{C}\boxtimes\mathfrak{C}^{mop}$-module 2-functor. The proof of theorem 1 of \cite{JF} outlines that the multifusion 2-category $\mathfrak{C}$ yields a 4-dualizable object of $\mathbf{Mor}_1(\mathbf{Cau2Cat}_{\mathds{k}})$ if and only if there exists a left $\mathfrak{C}\boxtimes\mathfrak{C}^{mop}$-module 2-functor $\Delta:\mathfrak{C}\rightarrow \mathfrak{C}\boxtimes\mathfrak{C}^{mop}$ such that the composite left $\mathfrak{C}\boxtimes\mathfrak{C}^{mop}$-module 2-endofunctor $T\circ \Delta$ on $\mathfrak{C}$ can be extended to a 3-condensation in the sense of \cite{GJF} with splitting $Id_{\mathfrak{C}}$ as a $\mathfrak{C}\boxtimes\mathfrak{C}^{mop}$-module 2-functor.

We now argue that this is possible given that the multifusion 2-category $\mathfrak{C}$ is separable. Through the equivalence of left $\mathfrak{C}\boxtimes\mathfrak{C}^{mop}$-module 2-categories of construction \ref{con:canonicalalgebra}, the 2-functor $T$ is identified with the left $\mathfrak{C}\boxtimes\mathfrak{C}^{mop}$-module 2-functor $\mathfrak{C}\boxtimes\mathfrak{C}^{mop}\rightarrow \mathbf{Mod}_{\mathfrak{C}\boxtimes\mathfrak{C}^{mop}}(\mathcal{R}_{\mathfrak{C}})$ given by $H\mapsto H\Box \mathcal{R}_{\mathfrak{C}}$. We let $\Delta$ be the forgetful 2-functor $\mathbf{Mod}_{\mathfrak{C}\boxtimes\mathfrak{C}^{mop}}(\mathcal{R}_{\mathfrak{C}})\rightarrow \mathfrak{C}\boxtimes\mathfrak{C}^{mop}$. Then, the composite $T\circ\Delta:\mathbf{Mod}_{\mathfrak{C}\boxtimes\mathfrak{C}^{mop}}(\mathcal{R}_{\mathfrak{C}})\rightarrow \mathbf{Mod}_{\mathfrak{C}\boxtimes\mathfrak{C}^{mop}}(\mathcal{R}_{\mathfrak{C}})$ is identified with the 2-functor $H\mapsto H\Box \mathcal{R}_{\mathfrak{C}}$. But, theorem 5.1.2 of \cite{D8} exhibits an equivalence between the Morita 3-category of separable algebras in $\mathfrak{C}\boxtimes\mathfrak{C}^{mop}$ and the 3-category of separable left $\mathfrak{C}\boxtimes\mathfrak{C}^{mop}$-module 2-categories. Under this equivalence, the composite left $\mathfrak{C}\boxtimes\mathfrak{C}^{mop}$-module 2-functor $T\circ\Delta$ corresponds to the $\mathcal{R}_{\mathfrak{C}}$-$\mathcal{R}_{\mathfrak{C}}$-bimodule $\mathcal{R}_{\mathfrak{C}}\Box \mathcal{R}_{\mathfrak{C}}$. Conversely, the identity left $\mathfrak{C}\boxtimes\mathfrak{C}^{mop}$-module 2-functor $Id_\mathfrak{C}$ corresponds to the $\mathcal{R}_{\mathfrak{C}}$-$\mathcal{R}_{\mathfrak{C}}$-bimodule $\mathcal{R}_{\mathfrak{C}}$. Finally, as $\mathfrak{C}$ is separable, then $\mathcal{R}_{\mathfrak{C}}$ is a separable algebra, so that the $\mathcal{R}_{\mathfrak{C}}$-$\mathcal{R}_{\mathfrak{C}}$-bimodule $\mathcal{R}_{\mathfrak{C}}\Box \mathcal{R}_{\mathfrak{C}}$ can be extended to a 3-condensation with splitting $\mathcal{R}_{\mathfrak{C}}$. We therefore find that $T\circ\Delta$ can be extended to a 3-condensation with splitting $Id_{\mathfrak{C}}$, which proves the desired result.
\end{Remark}

We explain how to categorify the notion of the global (also called categorical) dimension of a fusion 1-category. If $\mathfrak{C}$ is a fusion 2-category, we have obtained in construction \ref{con:canonicalalgebra} a connected rigid algebra $\mathcal{R}_{\mathfrak{C}}$ in $\mathfrak{C}\boxtimes\mathfrak{C}^{mop}$. Then, via the construction given in section 3.2 of \cite{D7}, one can associate to $\mathcal{R}_{\mathfrak{C}}$ a well-defined scalar $\mathrm{Dim}_{\mathfrak{C}\boxtimes\mathfrak{C}^{mop}}(\mathcal{R}_{\mathfrak{C}})$, called its dimension.

\begin{Definition}
The dimension of a fusion 2-category $\mathfrak{C}$ is $$\mathrm{Dim}(\mathfrak{C}):= \mathrm{Dim}_{\mathfrak{C}\boxtimes\mathfrak{C}^{mop}}(\mathcal{R}_{\mathfrak{C}}).$$
\end{Definition}

Thanks to theorem \ref{thm:algebrarigidseparable}, the rigid algebra $\mathcal{R}_{\mathfrak{C}}$ is actually separable. Combining this observation with theorem 3.2.4 of \cite{D7} yields the following result, which categorifies theorem 2.3 of \cite{ENO1}.

\begin{Corollary}\label{cor:nonzerodimension}
The dimension $\mathrm{Dim}(\mathfrak{C})$ of any fusion 2-category $\mathfrak{C}$ is non-zero. 
\end{Corollary}

\begin{Remark}
A notion of dimension for finite semisimple 2-categories was introduced in definition 1.2.28 of \cite{DR}. We do not know how this notion compares with ours in general, though these two notions agree in the examples considered below. However, we point out that there is a priori no reason to expect that they agree, as the definition given in \cite{DR} does not take into account the rigid monoidal structure.
\end{Remark}

\begin{Example}\label{ex:bosonicstronglyfusionseparable}
Let $\mathfrak{C}$ be a bosonic strongly fusion 2-category, so that $\mathfrak{C}$ is equivalent to $\mathbf{2Vect}_G^{\pi}$ for some finite group $G$ and 4-cocycle $\pi$ for $G$ with coefficients in $\mathds{k}^{\times}$. Without loss of generality, we set $\mathfrak{C}:=\mathbf{2Vect}_G^{\pi}$. Then, direct inspection shows that we have $$\mathfrak{C}\boxtimes\mathfrak{C}^{mop}= \mathbf{2Vect}_G^{\pi}\boxtimes (\mathbf{2Vect}_G^{\pi})^{mop}\simeq \mathbf{2Vect}_{G\times G^{op}}^{\pi\times \pi^{op}}$$ as fusion 2-categories. By construction \ref{con:canonicalalgebra}, we have that $\mathfrak{C}\simeq \mathbf{Mod}_{\mathfrak{C}\boxtimes\mathfrak{C}^{mop}}(\mathcal{R}_{\mathfrak{C}})$. Now, given two elements $g,h\in G$, we write $\mathbf{Vect}_{(g,h)}$ for the simple object of $\mathbf{2Vect}_{G\times G^{op}}^{\pi\times \pi^{op}}$ given by $\mathbf{Vect}$ with grading $(g,h)$. It follows from the definition above that $$\mathcal{R}_{\mathfrak{C}} = \boxplus_{g\in G} \mathbf{Vect}_{(g, g^{-1})}$$ as an object of $\mathfrak{C}\boxtimes\mathfrak{C}^{mop}$. On the other hand, let $\mathfrak{D}$ be the fusion sub-2-category of $\mathbf{2Vect}_{G\times G^{op}}^{\pi\times \pi^{op}}$ generated by the simple objects of the form $\mathbf{Vect}_{(g, g^{-1})}$ with $g\in G$. It is clear that $\mathfrak{D}\simeq\mathbf{2Vect}_G$, and that $\mathcal{R}_{\mathfrak{C}}$ is a connected rigid algebra in $\mathfrak{D}$. It therefore follows from corollary 3.3.7 of \cite{D7} that $\mathcal{R}_{\mathfrak{C}}$ has dimension $|G|$, so that $\mathrm{Dim}(\mathbf{2Vect}_G^{\pi})=|G|$. In this case, we note that the dimension of $\mathbf{2Vect}_G^{\pi}$ in the sense of \cite{DR} is also $|G|$.
\end{Example}

\begin{Example}\label{ex:dimensionconnectedfusion}
Let $\mathcal{B}$ be a braided fusion 1-category, and write $\mathrm{dim}(\mathcal{B})$ for its global dimension. We now compute the dimension of the connected fusion 2-category $\mathbf{Mod}(\mathcal{B})$. In order to do so, we consider $\mathbf{Mod}(\mathcal{B})$ as a left module 2-category over $\mathbf{Mod}(\mathcal{B})\boxtimes \mathbf{Mod}(\mathcal{B})^{mop}\simeq \mathbf{Mod}(\mathcal{B}\boxtimes\mathcal{B}^{rev})$. In this case, we find that the canonical rigid algebra $\mathcal{R}_{\mathbf{Mod}(\mathcal{B})}$ in $\mathbf{Mod}(\mathcal{B}\boxtimes\mathcal{B}^{rev})$ is given by $\mathcal{B}$ equipped with its canonical $\mathcal{B}\boxtimes\mathcal{B}^{rev}$-central structure. In particular, it follows from the refinement of proposition 3.3.3 of \cite{D7} given in proposition 3.2.3.3 of \cite{D:thesis} that $$\mathrm{Dim}(\mathbf{Mod}(\mathcal{B})) = 1/\mathrm{dim}(\mathcal{B}).$$ In particular, for connected fusion 2-categories, our notion of dimension coincides with that introduced in \cite{DR}. Further, the above equality shows that the dimension of a fusion 2-category is not invariant under Morita equivalence. Namely, for any finite group $G$, we have $\mathrm{Dim}(\mathbf{2Rep}(G)) = 1/|G|$, but we have seen above that $\mathrm{Dim}(\mathbf{2Vect}_G) = |G|$.
\end{Example}

\begin{Remark}
Let $\mathfrak{C}$ be a fusion 2-category, and $A$ a connected rigid algebra in $\mathfrak{C}$. Let us write $\mathfrak{D}:=\mathbf{Bimod}_{\mathfrak{C}}(A)$. We wonder whether the following equality always holds $$\frac{\mathrm{Dim}(\mathfrak{C})}{\mathrm{Dim}(\mathfrak{D})} = (\mathrm{Dim}_{\mathfrak{C}}(A))^2.$$ Namely, it holds in all of the above examples.
\end{Remark}

\subsection{Finite Semisimplicity of the Drinfeld Center}

In the previous section, we have associated a canonical separable algebra $\mathcal{R}_{\mathfrak{C}}$ to any multifusion 2-category $\mathfrak{C}$ such that $\mathfrak{C}\simeq \mathbf{Mod}_{\mathfrak{C}\boxtimes\mathfrak{C}^{mop}}(\mathcal{R}_{\mathfrak{C}})$ as left $\mathfrak{C}\boxtimes\mathfrak{C}^{mop}$-module 2-categories. It then follows from theorem 5.1.2 of \cite{D8} and lemma \ref{lem:module2funcenter} above that $$\mathbf{Bimod}_{\mathfrak{C}\boxtimes\mathfrak{C}^{mop}}(\mathcal{R}_{\mathfrak{C}})\simeq \mathscr{Z}(\mathfrak{C})$$ as monoidal 2-categories. In particular, by appealing to theorem \ref{thm:bimodulefinitesemisimple}, we get the following corollary. 

\begin{Corollary}\label{cor:F2Ccenter}
The Drinfeld center of any multifusion 2-category is a finite semisimple 2-category.
\end{Corollary}

We end by recording some additional properties of the Drinfeld center of a fusion 2-category, which follow from the above theorems.

\begin{Corollary}\label{cor:center2Deligne}
For every two multifusion 2-categories $\mathfrak{C}$ and $\mathfrak{D}$, we have $\mathscr{Z}(\mathfrak{C}\boxtimes\mathfrak{D})\simeq \mathscr{Z}(\mathfrak{C})\boxtimes \mathscr{Z}(\mathfrak{D})$ as braided monoidal 2-categories.
\end{Corollary}
\begin{proof}
Let us write $\mathfrak{E}:= \mathfrak{C}\boxtimes \mathfrak{C}^{mop}\boxtimes \mathfrak{D}\boxtimes \mathfrak{D}^{mop}\simeq \mathfrak{C}\boxtimes\mathfrak{D}\boxtimes (\mathfrak{C}\boxtimes\mathfrak{D})^{mop}$. It follows from combining construction \ref{con:canonicalalgebra} and proposition 4.1 of \cite{D3} that the separable algebra $\mathcal{R}_{\mathfrak{C}\boxtimes \mathfrak{D}}$ in $\mathfrak{E}$ is equivalent to $\mathcal{R}_{\mathfrak{C}}\boxtimes \mathcal{R}_{\mathfrak{D}}$. In particular, there is an equivalence of monoidal 2-categories  $\mathscr{Z}(\mathfrak{C}\boxtimes \mathfrak{D})\simeq\mathbf{Bimod}_{\mathfrak{E}}(\mathcal{R}_{\mathfrak{C}}\boxtimes \mathcal{R}_{\mathfrak{D}})$. Furthermore, it follows from lemma \ref{lem:bimodule2Deligne} that there is an equivalence $\mathbf{Bimod}_{\mathfrak{E}}(\mathcal{R}_{\mathfrak{C}}\boxtimes \mathcal{R}_{\mathfrak{D}})\simeq \mathscr{Z}(\mathfrak{C})\boxtimes\mathscr{Z}(\mathfrak{D})$ of monoidal 2-categories. Finally, one checks that the composite monoidal 2-functor $\mathscr{Z}(\mathfrak{C})\boxtimes\mathscr{Z}(\mathfrak{D})\simeq\mathbf{Bimod}_{\mathfrak{E}}(\mathcal{R}_{\mathfrak{C}}\boxtimes \mathcal{R}_{\mathfrak{D}})\simeq \mathscr{Z}(\mathfrak{C}\boxtimes \mathfrak{D})$ agrees with the canonical braided monoidal 2-functor constructed in lemma \ref{lem:CenterProductWeak}. This finishes the proof of the result.
\end{proof}

Combining the above corollary with theorem \ref{thm:stronglyfusioninvertible} yields the following result.

\begin{Corollary}
The Drinfeld center of any fusion 2-category is equivalent as a braided fusion 2-category to the Drinfeld center of a strongly fusion 2-category.
\end{Corollary}

\begin{Remark}
We note that there are non-trivial braided monoidal equivalences between the Drinfeld centers of strongly fusion 2-categories. For instance, given a finite group $G$ and a 4-cocycle $\pi$ for $G$ with coefficients in $\mathds{k}^{\times}$, any outer automorphism $\alpha$ of $G$ induces an equivalence $\mathbf{2Vect}_G^{\pi}\simeq \mathbf{2Vect}_G^{\alpha^*(\pi)}$ of fusion 2-categories. In addition, such an equivalence can be twisted using a 3-cocycle for $G$ with coefficients in $\mathds{k}^{\times}$. Then, by taking the Drinfeld center, we get an induced braided monoidal equivalence $$\mathscr{Z}(\mathbf{2Vect}_G^{\pi})\simeq \mathscr{Z}(\mathbf{2Vect}_G^{\alpha^*(\pi)}).$$ We expect that every braided monoidal equivalence between the Drinfel centers of bosonic strongly fusion 2-categories arises this way. On the other hand, there are more exotic braided monoidal equivalences between the Drinfeld centers of fermionic strongly fusion 2-categories as is explained in \cite{JF3}. The case of $\mathbf{2SVect}$ is discussed above in remark \ref{rem:braidedequivalencescenter}.
\end{Remark}

\begin{Remark}
Let us write $\mathscr{M}$ for the set of Morita equivalence classes of fusion 2-categories. This is a commutative monoidal under the 2-Deligne tensor product. We use $\mathscr{M}^{\times}$ to denote the maximal subgroup of $\mathscr{M}$. It follows from lemma \ref{lem:invertiblecharacterization} and theorem \ref{thm:WittEquivalence} (see also example 5.4.6 of \cite{D8}) that $\mathscr{M}^{\times}$ is isomorphic to $\mathcal{W}$, the Witt group of non-degenerate braided fusion 1-categories. Further, let us use $\mathbf{ZF2C}$ to denote the set of braided monoidal equivalence classes of Drinfeld centers of fusion 2-categories. Theorem \ref{thm:MoritaInvarianceDrinfeldCenter} and corollary \ref{cor:center2Deligne} imply that the map $\mathscr{Z}:\mathscr{M}/\mathcal{W}\rightarrow \mathbf{ZF2C}$ sending a Morita equivalence class of fusion 2-categories to its Drinfeld center is well-defined. As a generalization of proposition 4.1 of \cite{JFR}, we  conjecture that the map $\mathscr{Z}:\mathscr{M}/\mathcal{W}\rightarrow \mathbf{ZF2C}$ is a bijection.
\end{Remark}

\begin{Corollary}
For every multifusion 2-category $\mathfrak{C}$, the forgetful 2-functor $\mathbf{F}:\mathscr{Z}(\mathfrak{C})\rightarrow \mathfrak{C}$ is dominant, i.e.\ every object of $\mathfrak{C}$ is the splitting of a 2-condensation monad in $\mathfrak{C}$ supported on an object in the image of $\mathbf{F}$.
\end{Corollary}
\begin{proof}
The monoidal 2-functor $\mathbf{F}$ is dominant if and only if $\mathfrak{C}$ is an indecomposable left $\mathscr{Z}(\mathfrak{C})$-module 2-category. But, the dual to $\mathscr{Z}(\mathfrak{C})$ with respect to $\mathfrak{C}$ is $\mathfrak{C}\boxtimes\mathfrak{C}^{mop}$, a fusion 2-category. The result thus follows from corollary 5.2.5 of \cite{D8}.
\end{proof}

\begin{Corollary}
Let $\mathfrak{C}$ be a multifusion 2-category, then $\mathscr{Z}(\mathfrak{C})$ is a factorizable braided multifusion 2-category, i.e.\ the canonical braided monoidal 2-functor $\mathscr{Z}(\mathfrak{C})\boxtimes \mathscr{Z}(\mathfrak{C})^{rev}\rightarrow\mathscr{Z}(\mathscr{Z}(\mathfrak{C}))$ is an equivalence.
\end{Corollary}
\begin{proof}
The multifusion 2-categories $\mathscr{Z}(\mathfrak{C})$ and $\mathfrak{C}\boxtimes\mathfrak{C}^{mop}$ are Morita equivalent, so that $\mathscr{Z}(\mathfrak{C})\boxtimes \mathscr{Z}(\mathfrak{C})^{rev}\simeq\mathscr{Z}(\mathscr{Z}(\mathfrak{C}))$ as braided monoidal 2-category thanks to theorem \ref{thm:MoritaInvarianceDrinfeldCenter}. Inspection shows that the braided monoidal 2-functor constructed in this proof agrees with the canonical one.
\end{proof}

\begin{Remark}
Let $\mathfrak{B}$ be a braided fusion 2-category. We use $\mathscr{Z}_{(2)}(\mathfrak{B})$ to denote its sylleptic center, as defined in section 5.1 of \cite{Cr}. We expect that a braided fusion 2-category $\mathfrak{B}$ is factorizable if and only if it is non-degenerate in the sense that there is an equivalence $\mathscr{Z}_{(2)}(\mathfrak{B})\simeq\mathbf{2Vect}$ of sylleptic monoidal 2-categories.
\end{Remark}

Our last result is related to conjecture 4.3 of \cite{JFR}. More precisely, we find that the quotient of the canonical inclusion of abelian groups $\mathcal{W}\hookrightarrow\mathcal{W}(\mathbf{Rep}(G))$ is identified with $H^4(G;\mathds{k}^{\times})$ as a set.

\begin{Proposition}\label{prop:bosonicWittgroups}
Let $G$ be a finite group. There are bijections of sets $$\mathcal{W}(\mathbf{Rep}(G))\cong \mathcal{W}\times H^4(G;\mathds{k}^{\times})\ \ and\ \ \widehat{\mathcal{W}}(\mathbf{Rep}(G))\cong \mathcal{W}\times H^4(G;\mathds{k}^{\times})/Out(G).$$
\end{Proposition}
\begin{proof}
Let $G_1$ and $G_2$ be finite groups. In addition, let $\pi_1$ be a 4-cocycle for $G_1$ with coefficients in $\mathds{k}^{\times}$, and $\pi_2$ be a 4-cocycle for $G_2$ with coefficients in $\mathds{k}^{\times}$. Let also $\mathcal{B}_1$ and $\mathcal{B}_2$ be non-degenerate braided fusion 2-categories. We define $$\mathfrak{C}_1:=\mathbf{2Vect}_{G_1}^{\pi_1}\boxtimes\mathbf{Mod}(\mathcal{B}_1)\ \ \mathrm{and}\ \ \mathfrak{C}_2:=\mathbf{2Vect}_{G_2}^{\pi_2}\boxtimes\mathbf{Mod}(\mathcal{B}_2).$$ We examine under what conditions the two fusion 2-categories $\mathfrak{C}_1$ and $\mathfrak{C}_2$ are Morita equivalent.

It follows from corollary \ref{cor:center2Deligne}, theorem \ref{thm:MoritaInvarianceDrinfeldCenter}, and theorem 1.1 of \cite{KTZ} that if $\mathfrak{C}_1$ and $\mathfrak{C}_2$ are Morita equivalent, then $G_1\cong G_2$. Furthermore, there must exist a connected separable algebra $A$ in $\mathfrak{C}_1$ and a equivalence of fusion 2-categories $\mathbf{F}:\mathbf{Bimod}_{\mathfrak{C}_1}(A)\simeq\mathfrak{C}_2.$ The algebra $A$ is necessarily contained in $\mathfrak{C}_1^0$. Namely, if it was not, then it would follow from lemma \ref{lem:criterionbimoduleconnected} that $\mathbf{Bimod}_{\mathfrak{C}_1}(A)$ has strictly fewer than $|G_2|$ connected components. Thus, we find that $A$ is an algebra in $\mathbf{Mod}(\mathcal{B}_1)$, so that, together with $\mathbf{F}^0$, they witness a Witt equivalence between $\mathcal{B}_1$ and $\mathcal{B}_2$. As a consequence, we find that $$\mathfrak{C}_1\boxtimes\mathbf{Mod}(\mathcal{B}_1^{rev})\ \ \mathrm{and}\ \ \mathfrak{C}_2\boxtimes\mathbf{Mod}(\mathcal{B}_1^{rev})$$ are Morita equivalent. But, using the fact that $\mathbf{Mod}(\mathcal{B}_1)\boxtimes\mathbf{Mod}(\mathcal{B}_1^{rev})$ is Morita equivalent to $\mathbf{2Vect}$, we see that the above fusion 2-categories are Morita equivalent to $\mathbf{2Vect}_{G_1}^{\pi_1}$, and $\mathbf{2Vect}_{G_2}^{\pi_2}$ respectively. In particular, $\mathbf{2Vect}_{G_1}^{\pi_1}$ and $\mathbf{2Vect}_{G_2}^{\pi_2}$ are Morita equivalent. As $G_1\cong G_2$, it follows readily from lemmas \ref{lem:criterionbimoduleconnected} and \ref{lem:dualconnectedcomponent} that such a Morita equivalence has to be implemented by an equivalence of fusion 2-categories $$\mathbf{2Vect}_{G_1}^{\pi_1}\simeq \mathbf{2Vect}_{G_2}^{\pi_2},$$ that is, there exists an isomorphism of groups $f:G_1\cong G_2$ such that $f^*\pi_2$ is cohomologous to $\pi_1$.

Conversely, if there exists an isomorphism of groups $f:G_1\cong G_2$ such that $f^*\pi_2$ is cohomologous to $\pi_1$, and $\mathcal{B}_1$ and $\mathcal{B}_2$ are Witt equivalent, then $\mathfrak{C}_1$ and $\mathfrak{C}_2$ are manifestly Morita equivalent. Combining theorems \ref{thm:stronglyfusioninvertible} and \ref{thm:Moritaconnected} together with corollary \ref{cor:Moritaclassesconnected} yields the second isomorphism. The first ones follows similarly by working with $\mathbf{Mod}(\mathbf{Rep}(G))$-central fusion 2-categories as introduced in remark \ref{rem:centralMoritaequivalence}. Namely, in this case, the $\mathbf{Mod}(\mathbf{Rep}(G))$-central structures completely determine the isomorphism $G_1\cong G_2$.
\end{proof}